\documentclass[12pt,english]{article}
\usepackage{babel}
\usepackage[cp1250]{inputenc}
\usepackage[T1]{fontenc}
\usepackage{amsfonts}
\usepackage{amsmath}
\usepackage{amsthm}
\usepackage{hyperref}
\usepackage{enumerate}
\usepackage{graphicx}
\usepackage{pictex}
\usepackage{esint}

\newcommand{\eqq}[2]{\begin{equation}  #1  \label{#2} \end{equation}    }

\newcommand{\hd}{\hspace{0.2cm}}
\newcommand{\no}{\noindent}
\newcommand{\m}[1]{\mbox{ #1}}
\newcommand{\ww}{\subseteq }
\newcommand{\Om}{\Omega}
\newcommand{\ddt}{\frac{d}{dt}}
\newcommand{\ddta}{\frac{d}{d \tau}}

\newcommand{\rr}{\mathbb{R}}

\newcommand{\ald}{\frac{\al}{2}}

\newcommand{\jd}{\frac{1}{2}}

\newcommand{\tad}{t^{\ald}}

\newcommand{\Oml}{\Om_{l}}
\newcommand{\Omlt}{\Oml(t)}
\newcommand{\Oms}{\Om_{s}}
\newcommand{\Omst}{\Oms(t)}
\newcommand{\vlt}{V_{l}(t)}
\newcommand{\vst}{V_{s}(t)}

\newcommand{\nk}[1]{ \left[ #1 \right]}

\newcommand{\intfjmb}{I^{1-\al}}
\newcommand{\cabt}{D^{\al}}
\newcommand{\ep}{\varepsilon}
\newcommand{\ti}[1]{\tilde{#1}}
\newcommand{\sma}{s^{-1}(a)}
\newcommand{\smx}{s^{-1}(x)}

\newcommand{\da}{D^{\al}}

\newcommand{\xik}{\xi^{2}}

\newcommand*{\al}{\alpha}

\newcommand{\dasx}{D^{\al}_{  {\scriptsize{\smx}}  }}
\newcommand{\gja}{\Gamma(1-\al)}
\newcommand{\jgja}{\frac{1}{\gja}}
\newcommand{\ttaa}{(t-\tau)^{-\al}}
\newcommand{\la}{\lambda}
\newcommand{\ul}{u^{\la}}
\newcommand{\laa}{\la^{a}}
\newcommand{\lab}{\la^{b}}
\newcommand{\lac}{\la^{c}}
\newcommand{\lama}{\la^{-a}}
\newcommand{\lamb}{\la^{-b}}
\newcommand{\lamc}{\la^{-c}}
\newcommand{\lamba}{\la^{-b\al}}
\newcommand{\lamda}{\lambda^{-2a}}
\newcommand{\sm}{s^{-1}}

\newcommand{\xmda}{x^{-\frac{2}{\al}}}
\newcommand{\fda}{\frac{2}{\al}}
\newcommand{\fdak}{\left(\fda\right)^{2}}
\newcommand{\cjmda}{c_{1}^{-\fda}}
\newcommand{\xda}{x^{\fda}}
\newcommand{\icxi}{\int_{c_{0}}^{\xi}}
\newcommand{\xipal}{(\xi - p )^{-\al}}
\newcommand{\dakda}{\left[\fdak + \fda\right]}
\newcommand{\xima}{\frac{(\xi - c_{0})^{-\al}}{\gja}}
\newcommand{\fad}{\frac{\al}{2}}
\newcommand{\fadk}{\left( \frac{\al}{2} \right)^{2}}
\newcommand{\jga}{\frac{1}{\Gamma(\al)}}
\newcommand{\ga}{\Gamma(\al)}
\newcommand*{\norm}[1]{\left\Vert{#1}\right\Vert}
\newcommand*{\abs}[1]{\left\vert{#1}\right\vert}

\newcommand{\icz}{I_{c_{0}}}
\newcommand{\iczk}{I_{c_{0}}^{2}}
\newcommand{\ida}{I^{2-\al}_{c_{0}}}
\newcommand{\ija}{I^{1-\al}_{c_{0}}}
\newcommand{\inczxi}{\int_{c_{0}}^{\xi }}
\newcommand{\dakmda}{\left[ \fdak - \fda \right]}
\newcommand{\tdakmda}{\left[ 3\fdak - \fda \right]}
\newcommand{\ximk}{\xi^{-2}}

\newcommand{\Lcj}{L_{c_{1}}}
\newcommand{\Gcj}{G_{c_{1}}}
\newcommand{\inycj}{\int_{y}^{c_{1}}}
\newcommand{\inyp}{\int_{y}^{p}}
\newcommand{\fGG}{\frac{\Gamma(1+\fad)}{\Gamma(1-\fad)}}

\newcommand{\djma}{\partial^{1-\al}}
\newcommand{\djmasmx}{\djma_{\smx}}
\newcommand{\Dasmx}{D^{\al}_{\smx}}
\newcommand{\cjk}{\overline{c}_{1}}

\newcommand{\sal}{s_{\al}}
\newcommand{\ual}{u_{\al}}
\newcommand{\ualf}{\widetilde{u}_{\al}}
\newcommand{\uj}{u_{1}}
\newcommand{\tds}{t_{*}}
\newcommand{\tgs}{t^{*}}
\newcommand{\tjd}{t^{\frac{1}{2}}}
\newcommand{\sj}{s_{1}}
\newcommand{\Gcja}{G_{c_{1},\al}}
\newcommand{\Lcja}{L_{c_{1}, \al}}
\newcommand{\mja}{M_{1, \al}}
\newcommand{\mna}{M_{n, \al}}
\newcommand{\Gcjj}{G_{c_{1},1}}
\newcommand{\Lcjj}{L_{c_{1}, 1}}
\newcommand{\mjj}{M_{1, 1}}
\newcommand{\mnj}{M_{n, 1}}

\newcommand{\Gja}{\Gamma(1-\al)}
\newcommand{\lima}{\lim_{\al \nearrow 1}}
\newcommand{\sumn}{\sum_{n=0}^{\infty}}
\newcommand{\Qts}{Q_{\tds, \tgs}}
\newcommand{\tmad}{t^{-\ald}}
\newcommand{\tmjd}{t^{-\jd}}
\newcommand{\alz}{\al_{0}}

\newcommand{\smxa}{s^{-1}_{\al}(x)}
\newcommand{\smxj}{s^{-1}_{1}(x)}
\newcommand{\Dasmxa}{D^{\al}_{\smxa}}
\newcommand{\ualxx}{u_{\al, xx}}
\newcommand{\ualt}{u_{\al, t}}

\newcommand{\vp}{\varphi}

\newcommand{\insats}{\int_{\smxa}^{\tgs}}
\newcommand{\insjts}{\int_{\smxj}^{\tgs}}
\newcommand{\insat}{\int_{\smxa}^{t}}
\newcommand{\ta}{(t- \tau )^{-\al}}
\newcommand{\intats}{\int_{\tau}^{\tgs}}

\newtheorem{remark}{Remark}
\newtheorem{prop}{Proposition}
\newtheorem{corollary}{Corollary}
\newtheorem{theorem}{Theorem}

\begin{document}

\title{A self-similar solution to time-fractional Stefan problem}

\author{A. Kubica,  K. Ryszewska\\
\medskip\\
Department of Mathematics and Information Sciences\\
Warsaw University of Technology\\
pl. Politechniki 1, 00-661 Warsaw, Poland\\ \\
e-mail:
{\tt A.Kubica@mini.pw.edu.pl}\\
{\tt K.Ryszewska@mini.pw.edu.pl} \\
}

\maketitle
\begin{quote}
 \footnotesize
{\bf Abstract} We derive the fractional version of one-phase one-dimensional Stefan model. We assume that the diffusive flux is given by the time-fractional Riemann-Liouville derivative, i.e. we impose the memory effect in the examined model. Furthermore, we find a special solution to this problem.
\end{quote}



\section{Introduction}
The purpose of this paper is to study the process of changing the phase of medium, in which the diffusion exhibits non-local in time effects.
We are motivated by the paper~\cite{FV}, where the authors represent the non-locality in time, assuming that the diffusive flux is given in the form of time-fractional Riemann-Liouville derivative of temperature gradient, i.e.
\eqq{q^{*}(x,t)= - \djma  T_{x}(x,t). }{memflux}
Based on this assumption, the authors derived the sharp-interphase as well as the diffusive interphase fractional Stefan model. The sharp-interphase model obtained in \cite{FV} is characterized by the replacement of time derivative by fractional Caputo derivative. In last years several attempts to solve this problem have been done. In \cite{KRR} the authors proved the existence of weak solutions in non-cylindrical domain with fixed boundary. In \cite{Hopf} under suitable regularity assumptions the Hopf lemma was proven. It is also worth to mention the paper \cite{exa1} were the special solution to this problem was found. However, due to the lack of regularity results, the time-fractional Stefan problem with the Caputo derivative has not been solved.

It has been noticed already in \cite{VollSt} that the obtained diffusive-interphase model does not converge to the sharp one. This result encouraged the researchers to investigate the time-fractional Stefan model more deeply. In papers \cite{Cere}, \cite{Rosb} - \cite{Rose} the authors discussed other possible formulations of time-fractional Stefan problem and compare the formulas for special solutions.
In paper \cite{Ros} there is shown that the time-fractional sharp-interphase model obtained in \cite{FV} is not a consequence of the assumption (\ref{memflux}). Moreover, the authors obtained a new model based on (\ref{memflux}).
In this paper, we derive the sharp-interphase model with non-local flux given by (\ref{memflux}) under mild regularity assumptions. We arrive at the similar model as in \cite{Ros}, however we obtain additional boundary condition.
In both papers the model is derived from the law of energy conservation, however in \cite{Ros} only the liquid part of the domain is considered.
At last, we find a self-similar solution to this problem, which is the main result of this paper.

\section{Formulation of the problem}
In the paper we consider one-dimensional domain $\Om=(0,L)$, where $L$ is positive. We assume that at the initial time $t=0$ the domain $\Om$ is divided onto two parts: $(0,x_{0})$ - ``liquid'' and $(x_{0},L)$ - ``solid''. In particular, we admit the case where $x_{0}=0$. Following \cite{FV} we define the enthalpy function by $E= T+\phi$, where $T(x,t)$ is the temperature at point $x\in \Om$ at time $t$ and $\phi$ represents the latent heat.
We consider the sharp-interface model, hence we assume that $\phi$ is given in the following form
\eqq{\phi = \left\{ \begin{array}{ll}  1 & \m{ in liquid, }  \\
0  & \m{ in solid. }  \\ \end{array}  \right.}{a1}
We shall consider the one-phase model, i.e. we assume that $T\equiv 0$ in ``solid'' part. We denote by $q^{*}(x,t)$ the flux at $x\in \Om$ at time $t$. In this setting, the principle of energy conservation takes the following form: for every $V=(a,b)\ww \Om$
\eqq{\ddt \int_{V} E(x,t) dx  = q^{*}(a,t)- q^{*}(b,t). }{claw}

We may easily see that if the model does not exhibit memory effects then identity (\ref{claw}) leads to classical one-phase Stefan problem. We state this result in the remark.
\begin{remark}\label{remone}
If the flux is defined by the Fourier law $q^{*}(x,t)= -  T_{x}(x,t)$, then (\ref{claw}) leads to the classical Stefan problem
\eqq{\ddt T(x,t)- T_{xx}(x,t)=0  \hd \m{ for  } (x,t) \in [0,L]\times (0,\infty) \setminus \{(s(\tau),\tau): \hd \tau\in (0,\infty)\},  }{b2}
\eqq{\dot s(t) = - T_{x}^{-}(s(t),t)  \hd \m{ for  } \hd t>0,}{a6d}
where $s(t)$ denotes an interface and
\[
T_{x}^{-}(s(t),t) = \lim\limits_{\ep \rightarrow 0^{+}}T_{x}(s(t)-\ep,t).
\]
\end{remark}

In order to study non-local model we recall the definitions of fractional operators. By $I^{\al}_{a}$ we denote the fractional integral given by
\eqq{
I^{\al}_{a} f(t)= \frac{1}{\Gamma(\al)}\int_{a}^{t} (t- \tau)^{\al-1}f(\tau) d\tau.
}{defIala}
We also introduce the Riemann-Liouville and the Caputo fractional derivatives defined respectively by
\[
\partial^{\al}_{a}f(t) = \frac{d}{dt}I^{1-\al}_{a}f(t),\hd \hd D^{\al}_{a}f = \frac{d}{dt}I^{1-\al}_{a}[f(t)-f(a)] \hd \m{ for } \al \in (0,1).
\]
If a subscript $a = 0$ we omit it in a notation.
Following \cite{FV}, we assume that the flux is given by the Riemann-Liouville fractional derivative with respect to the time variable, i.e.
\[
q^{*}(x,t)= - \djma  T_{x}(x,t),
\]
where
\[
\djma  T_{x}(x,t) = \frac{1}{\Gamma(\al)} \ddt \int_{0}^{t} (t - \tau)^{\al-1}T_{x}(x,\tau) d\tau, \hd \al\in (0,1).
\]
We finish this section with a formal justification, why such a form of the flux seems to be reasonable in the model exhibiting memory effects.
\begin{remark}
Let us denote by $s(\cdot)$ the phase interface. We decompose the domain $\Om$ on the solid and liquid part.
\[
\Omlt = (0, s(t)) \m{  - liquid}, \hd \hd \Omst = (s(t),L) \m{  - solid}.
\]
Let $V\ww \Om$ be arbitrary. Then, if we assume that $V=(a,b)$ and denote
\[
\vlt=\Omlt \cap V, \hd \hd \vst = \Omst \cap V
\]
then, (\ref{claw}) takes the form
\eqq{\ddt \nk{ \int_{\vlt } (T(x,t)+ 1) dx   }  + \ddt \nk{ \int_{\vst} T(x,t) dx  } = \djma   T_{x}(b,t) - \djma T_{x}(a,t). }{clawmem}
Assuming that the temperature gradient is bounded with respect to time variable, having integrated with respect to time  we arrive at
\[
\int_{\vlt } (T(x,t)+ 1) dx     +  \int_{\vst} T(x,t) dx = \int_{V_{l}(0) } (T(x,0)+ 1) dx     +  \int_{V_{s}(0)} T(x,0) dx
\]
\eqq{  +\frac{1}{\Gamma(\al)}  \int_{0}^{t} (t- \tau )^{\al-1} \nk{ T_{x}(b,\tau) - T_{x}(a,\tau)  } d \tau, }{d1}
i.e. the total enthalpy in $V$ at time $t$ is a sum of the initial enthalpy and the time-average of differences of local fluxes at the endpoints of $V$.

\end{remark}

\section{Main results}
We derive  the fractional Stefan model from the balance law (\ref{claw}) with the diffusive flux given by (\ref{memflux}). In order to do it rigorously we have to impose certain regularity conditions on the interface $s$ and the temperature function $T$.
\no The standard setting of the initial-boundary condition for the Stefan problem is the following
\[
T(x,0)=T_{0}(x) \geq 0 \m{ and } \hd T(0,t)=T_{D}(t)\geq 0 \hd \m{ or } \hd T_{x}(0,t)=T_{N}(t) \leq 0.
\]
We expect that if $T_{0}\equiv 0$, $T_{D}\equiv 0$ or $T_{0}\equiv 0$, $T_{N}\equiv 0$, then $T\equiv 0$. Otherwise, we expect
\eqq{\dot s (t)>0,\tag{A1}}{c1}
i.e.  melting of solid.

Subsequently, let us assume that $t^{*}$ is positive and
\[
s(t) \in AC[0,t^{*}], \hd  T_{x}(x,\cdot) \in AC[s^{-1}(x),t^{*}] \m{ for every } x \in \Omega,
\]
\begin{equation}
T_{x}(\cdot,t) \in AC[0,s(t)-\varepsilon] \m{ for every } \varepsilon > 0 \m{ and every } t \in (0,t^{*}),
\tag{A2}
\end{equation}
\[
T_{t}(\cdot, t)\in L^{1}(0,s(t)) \hd \m{ for each } t\in (0,t^{*}),
\]
where we denote
\[
Q_{s,t^{*}}=\{(x,t): \hd 0<x<s(t), \hd t\in (0,t^{*}) \}.
\]
Here and henceforth by $AC$ we denote the space of absolutely continuous functions.

We note that since we consider one-phase Stefan problem the temperature in the solid vanishes. Therefore, the flux is nonzero only in the liquid part of the domain, i.e. in $Q_{s,t^{*}}$ and it is given by the formula
\eqq{q^{*}(x,t)= \left\{  \begin{array}{cll} - \djmasmx  T_{x}(x,t)&   \m{ for }   & (x,t)\in Q_{s,t^{*}}, \\
0 &   \m{ for }   & (x,t)\not \in Q_{s,t^{*}},
\end{array}
 \right. }{memfluxe}
where
\eqq{\djmasmx   T_{x}(x,t) = \left\{ \begin{array}{lll}  \frac{1}{\Gamma(\al)} \ddt  \int_{0}^{t} (t- \tau)^{\al-1}  T_{x}(x,\tau) d \tau   \hd & \m{ for  } \hd & x\leq s(0), \\ \frac{1}{\Gamma(\al)} \ddt  \int_{s^{-1}(x)}^{t} (t- \tau)^{\al-1}  T_{x}(x,\tau) d \tau   \hd & \m{ for  } \hd & x> s(0). \\ \end{array}  \right.}{c2}

\no  This together with (\ref{a1}) leads to the following form of equality  (\ref{claw})
\eqq{\ddt \nk{ \int_{\vlt } T(x,t)+ 1 dx   }   = -q^{*}(b,t) + q^{*}(a,t). }{clawmems}
The last of the regularity assumptions, that we will make advantage of, are
\begin{equation}
 \dot{s}(t)\in L^{\infty}_{loc}((0,t^{\ast}]) \m{ and \hd }  \Dasmx   T (\cdot, t)\in L^{1}(0,s(t)) \m{ for } t \in (0,t^{*}),
\tag{A3}
\end{equation}

Now we are ready to formulate the first result of this paper.
Let us discuss the sharp one-phase, one-dimensional Stefan problem with the boundary condition $T(s(t),t)=0$.
\begin{theorem}\label{derive}
Under the assumptions (A1)-(A2), the conservation law (\ref{claw}) with  the flux given by (\ref{memflux}) leads to the following equation
\eqq{ \Dasmx   T(x,t)-T_{xx}(x,t)= \left\{  \begin{array}{cll} 0 & \m{ for } & x<s(0) \\  - \frac{1}{\Gamma(1-\al)} (t-s^{-1}(x))^{-\al}  &  \m{ for } &  x\in (s(0), s(t)) \end{array}\right.}{md8c}
for a.a. $(x,t)\in Q_{s,t^{\ast}}$, where
\eqq{\Dasmx   T(x,t) = \left\{ \begin{array}{lll}  \frac{1}{\Gamma(1-\al)}   \int_{0}^{t} (t- \tau)^{-\al}  \ddta T(x,\tau) d \tau   \hd & \m{ for  } \hd & x\leq s(0) \\ \frac{1}{\Gamma(1-\al)}   \int_{s^{-1}(x)}^{t} (t- \tau)^{-\al} \ddta T(x,\tau) d \tau   \hd & \m{ for  } \hd & x> s(0). \\ \end{array}  \right.}{d2}
Moreover, functions $T$ and $s$ are related by the formula
\eqq{
\dot s(t) =  - \frac{1}{\Gamma(\al)}\lim_{a\nearrow s(t)} \left[  \ddt \int_{\sma}^{t} (t-\tau)^{\al-1}T_{x}(a,\tau) d \tau\right].}{mnewabc}
Furthermore, if $(A3)$ holds, then the additional boundary condition
\eqq{ T^{-}_{x}(s(t),t)=0, }{md9c}
is satisfied, where $T^{-}_{x}$ is defined as in Remark~\ref{remone}.
\end{theorem}

\begin{remark}
We note that the equation (\ref{md8c}) with the condition (\ref{mnewabc}) have been already obtained in \cite{Ros}. It is worth to mention that the fractional Stefan problem with the flux given by the Riemann-Liouville derivative were considered in \cite{FV}. However, the Authors obtained the following system of equations
\eqq{\da   T(x,t)dx-T_{xx}(x,t)=0,}{Voll1}
\eqq{\da s(t) = - T_{x}(s(t),t),}{Voll2}
(see (17)and (18) in \cite{FV}). As pointed out in \cite{Ros} and proved by careful calculations, the equations (\ref{Voll1}) and (\ref{Voll2}) are not the consequences of the assumptions imposed on the flux.

In this paper we present another derivation of (\ref{md8c}) and (\ref{mnewabc}), which leads to the additional boundary condition (\ref{md9c}). Then, the following questions arise:
\begin{itemize}
\item
isn't the assumption (A3) to strong and the "unexpected" boundary condition (\ref{md9c}) is not appropriate?
\item
is there any relation between (\ref{mnewabc}) and (\ref{md9c})?
\end{itemize}
We partially answer to these questions. We show that, at least in the class of self-similar solutions, (A3) is satisfied and (\ref{mnewabc}) implies (\ref{md9c}) and conversely.
\end{remark}

\begin{remark}
Passing formally with $\al$ to $1$ in equations (\ref{md8c}) and (\ref{mnewabc}) we arrive at (\ref{b2})~-~ (\ref{a6d}). Indeed, assuming that $T(s(t),t) = 0$ we may write
\[
D^{\al}_{s^{-1}(x)}T(x,t) + \frac{1}{\Gamma(1-\al)}(t-s^{-1}(x))^{-\al} = \partial^{\al}_{s^{-1}(x)}[T(x,t)+1].
\]
Then,
\[
\partial^{\al}_{s^{-1}(x)}[T(x,t)+1]\rightarrow T_{t}(x,t) \m{ as } \al\rightarrow 1.
\]
Moreover, by (\ref{mnewabc}) and
\[
\lim_{a\rightarrow s(t)}\partial^{1-\al}_{s^{-1}(a)}T_{x}(a,t)\rightarrow T_{x}^{-}(s(t),t) \m{ as }\al\rightarrow 1
\]
we arrive at (\ref{a6d}). We shall examine these convergence more rigorously in section~\ref{classical} and we show that at  least in the class of self-similar solution, the additional boundary condition (\ref{md9c}) disappears for $\al =1$. This phenomenon seems to be quite interesting, because for $\al\in (0,1)$ conditions (\ref{mnewabc}) and (\ref{md9c}) are equivalent.

\end{remark}
Now we will present the second result of this paper. We will find a self-similar solution to the time-fractional Stefan problem in the domain
\eqq{U=\{(x,t)\in \rr \times (0,\infty) : \hd 0<x< s(t) \}, }{laa}
where $(s(t),t)$ is the curve separating the phases.  We impose a constant positive Dirichlet boundary condition on the left boundary and we assume that $s(0)=0$. In this case, the problem formulated in Theorem \ref{derive} takes the following form
\eqq{\dasx u(x,t) = u_{xx}(x,t) - \jgja (t - \smx )^{-\al} \hd \m{ in } \hd  U, }{lab}
\eqq{u(s(t),t) =0,}{kds}
\eqq{u(0,t) = \gamma,}{diri}
\eqq{\dot{s}(t) = - \jga  \lim_{a \nearrow s(t)} \ddt \left[ \int_{\sm (a)}^{t} (t - \tau)^{\al - 1} u_{x}(a, \tau ) d\tau   \right].  }{frees}

\begin{theorem}
\label{samo}
For any constant $\gamma > 0$ there exists a pair $(u,s)$ which satisfies (\ref{lab})-(\ref{frees}). Furthermore, the solution is given by
\eqq{s(t)= c_{1} t^{\fad},}{selfst}

\eqq{u(x,t)=\int_{xt^{-\fad}}^{c_{1}} H(p, xt^{-\fad})G_{c_{1}}(p)dp \hd \m{ in } \hd U,  }{selfuxt}
where $c_{1} = c_{1}(\al,\gamma) > 0$ and
\eqq{\Gcj (y) = \jgja \inycj (1- c_{1}^{-\fda} \mu^{\fda})^{-\al} d\mu \hd \m{ for } \hd  0\leq y \leq c_{1}, }{selfGcj}
\eqq{H(p,x) = 1+ \int_{x}^{p} N(p,y)dy \hd \m{ for } \hd 0\leq x \leq p,   }{selfdefH}
\eqq{N(p,y) = \sum_{n=1}^{\infty} M_{n}(p,y) \hd \m{ for } \hd 0\leq y \leq p , }{selfdefsumM}
where
\eqq{M_{1}(p,y) = \jgja \inyp (1- p^{-\fda } \mu^{\fda})^{-\al} d\mu \hd \m{ for } \hd 0\leq y \leq p  }{selfdefM}
and
\eqq{ M_{n}(p,y)= \int_{y}^{p} M_{1}(a,y)M_{n-1}(p, a) da \m{ for } \hd 0\leq y \leq p  \m{ and  } \hd n\geq 2. }{selfdefMn}
For every $R>0$ the series (\ref{selfdefsumM}) converges uniformly on $W_{R}=\{(p,y): \hd 0 \leq y\leq p \leq R \}$. Functions  $M_{n}$, $N$ are positive on $\{(p,y): \hd 0\leq y < p  \}$, hence $u$ is positive in $U$.

\no For every $a, \la >0$ function $u$ satisfies the scaling property
\eqq{u(x,t)=u(\laa x, \la^{\frac{2a}{\al}} t) }{selfself}
and
\eqq{u_{x}(s(t),t)= 0.}{kdt}
Furthermore, for every $t>0$ there holds  $u(\cdot , t)\in W^{2,1}(0,s(t))$ and $u_{t}(x,\cdot ) \in C([\sm(x), \infty))$ for every $x>0$. Finally, we have $u_{x}(x,\cdot) \in L^{\infty}(s^{-1}(x),\infty)\cap AC_{loc} ([s^{-1}(x),\infty))$ for every $x > 0$ and $u_{t}(\cdot,t) \in L^{1}(0,s(t))$, $D^{\al}_{s^{-1}(\cdot)}u(\cdot,t) \in L^{1}(0,s(t))$ for every $t > 0$. In particular, the pair $(u,s)$ satisfies the assumptions (A1) - (A3).
\end{theorem}

\begin{corollary}
\label{zcjgamma}
If $c_{1}$ is a positive constant and
\[
\gamma = \int_{0}^{c_{1}} H(p,0)G_{c_{1}}(p)dp,
\]
then  (\ref{selfst})-(\ref{selfuxt}) define a solution to (\ref{lab})-(\ref{frees}).
\end{corollary}

\begin{corollary}\label{neum}
If the Dirichlet condition (\ref{diri}) is replaced by the Neumann condition
\[
u_{x}(0,t) = -\beta t^{-\frac{\al}{2}}, \hd \beta > 0,
\]
then Theorem \ref{samo} holds with $c_{1} = c_{1}(\al,\beta) > 0$.
\end{corollary}

At last, we obtain the convergence of self-similar solutions to fractional Stefan problem to a solution to the classical Stefan problem. To formulate the result we introduce new notation. We fix $c_{1}>0$ and for $\al \in (0,1)$ we denote by $\sal$ and $\ual$ the solution to fractional Stefan problem ((\ref{lab}) - (\ref{frees})) given by (\ref{selfst}) and (\ref{selfuxt}). Then we set
\eqq{\ualf(x,t)= \left\{ \begin{array}{cll} \ual (x,t) & \m{ for } & \hd t>0, \hd x\in [0,\sal(t)], \\ 0 & \m{  for }  & \hd  t>0, \hd  x>\sal(t).  \\ \end{array}    \right.}{defualf}

\begin{theorem}
Let us fix $0<\tds< \tgs$.
If $\al \nearrow 1$, then $\ualf$ converges uniformly on the set $\{(x,t): \hd t\in [\tds, \tgs ], \hd x\in [0, c_{1} \tjd ]  \}$ to $\uj$, where
$\uj$ is a solution to the classical Stefan problem corresponding to the free boundary $s_{1}:=c_{1}\tjd$, i.e. $\sj$ and $\uj $ satisfy
\eqq{u_{1,t}(x,t) - u_{1,xx}(x,t)=0 \hd \m{ for  } \hd t>0, \hd x\in (0, \sj(t)),  }{classa}
\eqq{\uj(\sj(t),t)=0 \m{ \hd  for \hd }  t>0,}{classb}
\eqq{\uj(0,t) = 2ae^{a^{2}}\int_{0}^{a}e^{-w^{2}}dw\m{ \hd  for \hd }  t>0, \m{ \hd where\hd } a=\frac{c_{1}}{2},   }{classc}
\eqq{\ddt \sj(t)= - u_{1,x}(\sj(t), t ) \m{ \hd  for \hd }  t>0}{classd}
and $\uj $ is given by the formula
\eqq{\uj(x,t)=  2ae^{a^{2}}\int_{\frac{x}{2\sqrt{t}}}^{a}e^{-w^{2}}dw.}{classe}
\label{uniformly}
\end{theorem}

\section{Derivation of the model}

In this section we will prove Theorem \ref{derive}.

\begin{proof}[Proof of Theorem \ref{derive}]
In order to derive the system of equations from (\ref{clawmems}), we apply the principle of energy conservation to an arbitrary subset $V$ of the domain at time $t \in (0,t^{*})$. We will consider two cases.
\begin{itemize}
\item
If $V=(a,b)\ww (0,s(0))$,  then from (A1) we have $V \ww (0,s(t))$ for each $t\in (0,t^{*})$ and (\ref{clawmems}) gives
\[
\ddt \nk{ \int_{V} T(x,t)+1 dx}=\djma   T_{x}(b,t) - \djma T_{x}(a,t).
\]
Hence,
\[
  \int_{V} \ddt T(x,t) dx=\djma  T_{x}(b,t) - \djma T_{x}(a,t).
\]
We apply the fractional integral $\intfjmb$ with respect to the time variable to both sides of the identity and with a use of assumption (A2) and \cite[Theorem 2.4]{Samko} we arrive at
\[
\int_{V} \cabt T(x,t) dx = T_{x}(b,t) - T_{x}(a,t).
\]
By the fundamental theorem of calculus we obtain
\[
\int_{V} [ \cabt T(x,t) -  T_{xx}(x,t)]dx =0.
\]
Since $V \ww (0,s(0)) $ is arbitrary, we get
\eqq{\cabt T(x,t) -  T_{xx}(x,t) =0 \hd \hd \m{ for } \hd (x,t)\in (0,s(0))\times (0,t^{*}).}{a2es}
\item
If $V=(a,b)$, where $s(0)<a<s(t)<b$, then (\ref{clawmems}) has the form
\[
\ddt \nk{ \int_{a}^{s(t)} T(x,t)+1 dx} = q^{*}(a,t) = - \frac{1}{\Gamma(\al)} \ddt \int_{\sma}^{t} (t-\tau)^{\al-1}T_{x}(a,\tau) d \tau.
\]
Differentiating  the integral on the left hand side leads to
\[
\int_{a}^{s(t)} \ddt T(x,t) dx+\dot s(t)[T(s(t),t)+1] =  - \frac{1}{\Gamma(\al)} \ddt \int_{\sma}^{t} (t-\tau)^{\al-1}T_{x}(a,\tau) d \tau.
\]
Applying $T(s(t),t)=0$, we get
\eqq{
\int_{a}^{s(t)} \ddt T(x,t) dx+\dot s(t) =  - \frac{1}{\Gamma(\al)} \ddt \int_{\sma}^{t} (t-\tau)^{\al-1}T_{x}(a,\tau) d \tau.}{newaa}
If $a \nearrow s(t)$, then by the assumption $(A2)$ the first term vanishes and as a consequence we get (\ref{mnewabc}). Next, if we apply the operator $I^{1-\al}_{\sma}$ (defined in (\ref{defIala})) to both sides of (\ref{newaa}),  then  we obtain
\[
\frac{1}{\Gamma(1-\al)} \int_{\sma}^{t}(t-\tau)^{-\al} \int_{a}^{s(\tau)} \ddta T(x,\tau) dxd\tau+\frac{1}{\Gamma(1-\al)} \int_{\sma}^{t}(t-\tau)^{-\al} \dot s(\tau)d\tau
\]
\eqq{
=  - \frac{1}{\Gamma(\al)}\frac{1}{\Gamma(1-\al)} \int_{\sma}^{t}(t-\tau)^{-\al} \ddta \int_{\sma}^{\tau} (\tau-p)^{\al-1}T_{x}(a,p) dp d \tau.
}{dj}
We note that by the assumption (A2) we have $T_{x}(a, \cdot ) \in AC[s^{-1}(x),t^{*}]$ hence, applying \cite[Theorem 2.4]{Samko}, we may rewrite the right hand side of (\ref{dj}) as follows
\[
-I^{1-\al}_{\sma } \partial^{1-\al}_{\sma } T_{x}(a, t )(t) = -T_{x}(a,t).
\]
If we apply the Fubini theorem to the first term  in (\ref{dj}), then  we arrive at the identity
\eqq{
\int_{a}^{s(t)} \Dasmx  T(x,t)dx+ \frac{1}{\Gamma(1-\al)} \int_{\sma}^{t}(t-\tau)^{-\al} \dot s(\tau)d\tau = - T_{x}(a,t).
}{dc}
Applying the substitution $\tau = s^{-1}(x)$ we get
\[
\frac{1}{\Gamma(1-\al)} \int_{\sma}^{t}(t-\tau)^{-\al} \dot s(\tau)d\tau = \frac{1}{\Gamma(1-\al)} \int_{a}^{s(t)} (t-s^{-1}(x))^{-\al} dx.
\]
We expect that $T_{x}(\cdot, t)$ may admit singular behaviour near the phase change point. Thus, we proceed very carefully. We fix $\ep>0$ such that $a<s(t)-\ep$, then, by (A2) we have
\[
- T_{x}(a,t) = \int_{a}^{s(t)-\ep} T_{xx}(x,t)dx - T_{x}(s(t)-\ep,t).
\]
Making use of this identity in (\ref{dc}) we obtain
\[
\int_{a}^{s(t)-\ep} \nk{ \Dasmx   T(x,t)-T_{xx}(x,t)+\frac{1}{\Gamma(1-\al)} (t-s^{-1}(x))^{-\al} }dx
\]
\eqq{=-\int_{s(t)-\ep}^{s(t)} \nk{ \Dasmx   T(x,t)+\frac{1}{\Gamma(1-\al)} (t-s^{-1}(x))^{-\al} }dx -T_{x}(s(t)-\ep,t).  }{d3}
Let us choose arbitrary $\ti{a}$ such that $s(0)<\ti{a}<a$. Repeating the above calculations for $\ti{a}$ instead of $a$, we obtain that
\[
\int_{\ti{a}}^{s(t)-\ep} \nk{ \Dasmx   T(x,t)-T_{xx}(x,t)+\frac{1}{\Gamma(1-\al)} (t-s^{-1}(x))^{-\al} }dx
\]
\eqq{=-\int_{s(t)-\ep}^{s(t)} \nk{ \Dasmx   T(x,t)+\frac{1}{\Gamma(1-\al)} (t-s^{-1}(x))^{-\al} }dx -T_{x}(s(t)-\ep,t).  }{d33}
Subtracting the sides of (\ref{d3}) and (\ref{d33}) we arrive at
\eqq{\int_{\ti{a}}^{a} \nk{ \Dasmx   T(x,t)-T_{xx}(x,t)+\frac{1}{\Gamma(1-\al)} (t-s^{-1}(x))^{-\al} }dx =0 }{d12}
for arbitrary $a, \ti{a} \in (s(0),s(t)-\ep)$ hence, we may deduce that
\eqq{ \Dasmx   T(x,t)-T_{xx}(x,t)+\frac{1}{\Gamma(1-\al)} (t-s^{-1}(x))^{-\al} = 0 \hd \m{ for } \hd x\in (s(0), s(t)), }{d6}
i.e. (\ref{md8c}) is proven.

\no It remains to show (\ref{md9c}). From (\ref{d33}) and (\ref{d6}) we infer that
\[
0=-\int_{s(t)-\ep}^{s(t)} \nk{ \Dasmx   T(x,t)+\frac{1}{\Gamma(1-\al)} (t-s^{-1}(x))^{-\al} }dx -T_{x}(s(t)-\ep,t).
\]
In order to obtain additional information about $T_{x}(s(t),t)$, we employ further regularity assumptions.  Applying (A3) we immediately get
\eqq{ \lim_{\ep \rightarrow 0^{+}} \int_{s(t)-\ep}^{s(t)}(t-s^{-1}(x))^{-\al}dx=0 \hd
\m{and} \hd
\lim_{\ep \rightarrow 0^{+}} \int_{s(t)-\ep}^{s(t)}  \Dasmx   T(x,t)dx = 0.}{e2}

Making use of (\ref{e2}) we obtain
\eqq{\lim_{\ep \rightarrow 0^{+}} T_{x}(s(t)-\ep,t) = 0, }{e5}
hence, we arrive at (\ref{md9c}), which finishes the proof of Theorem~\ref{derive}.
\end{itemize}
\end{proof}

\section{Self-similar solution }
This section is devoted to the proof of Theorem \ref{samo}. The proof will be divided  into a few steps. At first we will proceed with formal calculations that will lead us to appropriate scaling.
We introduce parameters $a,b,c,\lambda >0$ and we define the function
\eqq{\ul (x,t) = \lac u(\laa x , \lab t).}{lad}
Our aim is to find $a,b,c$ and the curve $(s(t),t)$ such that, if $(u,s)$ is a solution to (\ref{lab}), then $\ul =u$.

\no At first, we perform calculations. We note that $u_{xx}(x,t)= \lamc \lamda \ul_{xx}(\lama x , \lamb t)$ and
\[
\gja \dasx u(x,t) = \int_{\smx}^{t} \ttaa u_{t}(x,\tau) d\tau = \lamc \lamb \int_{\smx}^{t} \ttaa \ul_{t}(\lama x , \lamb \tau) d\tau
\]
\[
= \lamc  \int_{\lamb \smx}^{ t \lamb } (t-\lab p )^{-\al} \ul_{t}(\lama x , p) dp=\lamc \lamba \int_{\lamb \smx}^{ t \lamb } (t\lamb - p )^{-\al} \ul_{t}(\lama x , p) dp
\]
\[
=\lamc \lamba \gja \da_{\lamb \smx} \ul (\lama x , \lamb t),
\]
i.e.
\[
\da_{\sm (\laa x)} u (\laa x , \lab x ) = \lamc \lamba \da_{\lamb \sm (\laa x)}\ul (x,t).
\]
Hence, if the pair $(u,s)$ is a solution to (\ref{lab}), then
\[
0 = \da_{\sm(\laa x )} u (\laa x, \lab t ) - u_{xx}(\laa x , \lab t )+ \jgja (\lab t - \sm (\laa x ))^{-\al}
\]
\[
=\lamc \lamba \da_{\lamb \sm (\laa x)} \ul ( x ,  t) - \lamc \lamda \ul_{xx}( x , t )
+ \jgja \lamba ( t - \lamb \sm (\laa x ))^{-\al}.
\]
Thus, if we set $c=0$ and
 \eqq{b=\frac{2a}{\al},}{lae}
then we get
\[
0=  \da_{\lamb \sm (\laa x )} \ul ( x ,  t) - \ul_{xx}( x , t ) +  \jgja  ( t - \lamb \sm (\laa x ))^{-\al}.
\]
We observe that, if $s(t)$ satisfies
\eqq{\smx = \lamb \sm (\laa x ),}{laf}
then $u$ and $\ul$ are the solutions to the same equation. From the identity (\ref{lae}) we infer that $\smx = \la^{-\frac{2a}{\al}} \sm (\laa x )$. Hence, the function $s^{-1}$ fulfills the functional equation  $g(\la x ) = \la^{\frac{2}{\al} } g(x)$. To solve this equation, it is enough to write
\[
\frac{g(x)-g(\la x )}{x(1-\la )} = \frac{g(x)}{x} \frac{1-\la^{\frac{2}{\al}}}{1-\la}
\]
and take the limit $\la \rightarrow 1$. Then we get that $g'=\frac{2}{\al} \frac{g}{x}$, i.e. $g(x) = cx^{\frac{2}{\al}}$. Thus, we obtained that, if there exists a self-similar solution, then the phase interface may have a form
\eqq{
s(t) =c_{1}t^{\frac{\al}{2}}
}{sself}
for some positive $c_{1}$.
If we denote
\eqq{c_{0} = \cjmda,}{cjczself}
then we may write
\eqq{\smx = c_{0} \xda.}{fsm}
Our aim is to find a special solution $u$ to the system (\ref{lab}), (\ref{kds}), (\ref{frees}),  when function $s$ is given by (\ref{sself}). We will proceed as follows. At first, we will rewrite the equations (\ref{lab}), (\ref{frees}) in terms of a new self-similar solution. Subsequently, we will show that in this setting, assuming appropriate regularity of $u$, condition (\ref{frees}) implies $u_{x}(s(t),t) =0$. Then, we will solve the problem
\eqq{\dasx u(x,t) = u_{xx}(x,t) - \jgja (t - \smx )^{-\al} \hd \m{ in } \hd U,}{labf}
\[
u(s(t),t) = 0, \hd u_{x}(s(t),t)=0 \hd \m{ for } \hd t>0,
\]
with $s$  given by (\ref{sself}). Then, we will show that the solution satisfies (\ref{frees}).
In the final section, we will prove that the obtained solution is positive and that for every $\gamma > 0$ we may find $c_{1} > 0$ such that the obtained solution satisfies Dirichlet boundary condition $u(0,t) = \gamma$.
\subsection{Similarity variable}
Let us begin with introducing a similarity variable
\eqq{\xi = t \xmda.}{defxi}
We define function $f$ as follows
\eqq{f(\xi)=f(t\xmda ):= u(x,t).}{deff}

In the next proposition we establish how the expected regularity properties of $u$ transforms to the properties of $f$. Furthermore, we will rewrite the conditions (\ref{labf}), (\ref{kds}), (\ref{frees}) in terms of $f$ and prove that (\ref{frees}) implies vanishing of derivative of $f$ in point $c_{0}$.
\begin{prop}
Let us assume that $s$ is given by (\ref{sself}) with fixed $c_{1} > 0$ and  $u$ is a self-similar solution to (\ref{lab}), (\ref{kds}), (\ref{frees}), where the similarity variable is given by (\ref{defxi}). Suppose that $u$ has following regularity.
For  $k > 1$ and every $t>0$ there hold  $u_{x}(\cdot , t)\in L^{1}(0, s(t))$,  $u_{xx}(\cdot , t)\in L^{1}(s(t)/k, s(t))$. Then, the function $f$ defined by (\ref{deff})  satisfies  $f'\in L^{1}(c_{0},\infty)\cap AC([c_{0}, k^{\frac{2}{\al}}c_{0}])$, $f\in C^{2}(c_{0}, k^{\frac{2}{\al}}c_{0})$ and for $\xi \in (c_{0}, k^{\frac{2}{\al}}c_{0})$ we have
\eqq{\jgja \icxi \xipal f'(p)dp = \fdak \xik f''(\xi ) + \dakda \xi f'(\xi ) - \xima, }{eqforf}
\eqq{f(c_{0})=0,  }{fini}
\eqq{
\fadk c_{0}^{-2} \ga  =\lim_{b \searrow c_{0}}  \frac{d}{db}  \left[ \int_{c_{0}}^{b} (b- p)^{\al -1 }   f'(p)   dp \right].
}{fxic}
The identity (\ref{eqforf}) together with regularity of $f$ implies
\eqq{\lim_{\xi \searrow c_{0}} (\xi-c_{0})^{\al}f''(\xi) = \fadk \frac{c_{0}^{-2}}{\gja }, }{fxid}
while from (\ref{fxic}) we deduce
\eqq{
f'(c_{0})=0.
}{fxidd}
\label{fprim}
\end{prop}
\begin{proof}
Let us begin with a simple calculation,
\eqq{u_{t}(x, \tau )= f'(\tau \xmda)\xmda,}{fxit}
\eqq{u_{x}(x,t)=-\fda f'(t\xmda)  t x^{-\fda -1 }, }{fxix}
\eqq{u_{xx}(x,t)=\fdak f''(t\xmda) ( t x^{-\frac{2}{\al}})^{2} x^{-2}+\fda(\fda +1) f'(t\xmda)  (t x^{-\fda  })x^{-2}.   }{fxixx}

Applying the substitution $p=\tau \xmda$ we get
\[
\dasx u(x,t)  = \jgja \int_{c_{0}\xda }^{t} (t-\tau)^{-\al} f'(\tau \xmda )\xmda d\tau
\]
\[
=\jgja \int_{c_{0}}^{t\xmda }(t-\xda p )^{-\al}f'(p) dp = x^{-2}\jgja \int_{c_{0}}^{t\xmda }(t\xmda - p )^{-\al}f'(p) dp.
\]
Furthermore, we have
\[
(t- c_{0}\xda)^{-\al} = x^{-2} (t\xmda - c_{0})^{-\al}.
\]

Applying these results in equation (\ref{lab}) with $s$ given by (\ref{sself}), we obtain (\ref{eqforf}).
To show that (\ref{fini}) holds, it is enough to notice that, since the function $u$ vanishes on the free boundary, we have
\[
0 =u(s(t), t) = u (c_{1} t^{\fad}, t ) = f(c_{0}),
\]
where we used (\ref{cjczself}). Now, we will prove the regularity results.
By (\ref{fxix}) we get
\eqq{
\infty>\int_{0}^{s(t)} |u_{x}(x,t)|dx =\fda\int_{0}^{s(t)} |f'(t\xmda)|  t x^{-\fda -1 }dx = \int_{c_{0}}^{\infty} |f'(\xi)|d\xi.
}{fxint}
From (\ref{fxixx}) we obtain in the similar way that
\[
\infty>\int_{s(t)/k}^{s(t)} |u_{xx}(x,t)|dx =\int_{s(t)/k}^{s(t)} \left| \fdak f''(t\xmda) ( t x^{-\frac{2}{\al}})^{2} x^{-2}+\fda(\fda +1) f'(t\xmda)  (t x^{-\fda  })x^{-2}\right| dx
\]
\[
= \int_{c_{0}}^{k^{\frac{2}{\al}}c_{0}} \left| \fda f''(\xi )\xi^{1+ \fad} t^{-\fad} + (\fda + 1 )f'(\xi) \xi^{\fad} t^{- \fad} \right| d\xi
\]
\[
\geq   \fda c^{1+\fad }_{0} t^{- \fad} \int_{c_{0}}^{k^{\frac{2}{\al}}c_{0}} \left|  f''(\xi ) \right| d\xi  - (\fda + 1 )k c_{0}^{\fad} t^{-\fad} \int_{c_{0}}^{\infty} |f'(\xi)|d\xi \m{ for every } t>0
\]
and as a consequence we obtain
\eqq{\int_{c_{0}}^{k^{\frac{2}{\al}}c_{0}} \left|  f''(\xi ) \right| d\xi <\infty. }{fxxint}
The estimates (\ref{fxint}) and (\ref{fxxint}) lead to $f'\in AC([c_{0}, k^{\frac{2}{\al}}c_{0}])$. Making use of the absolute continuity of $f'$ in identity (\ref{eqforf}) we deduce that $f\in C^{2}(c_{0}, k^{\frac{2}{\al}}c_{0})$. Hence, we obtained postulated regularity results.
Now, we shall rewrite the condition (\ref{frees}) in terms of the function $f$. We will show that it leads to (\ref{fxic}). Let us fix $a\in (s(t)/k, s(t))$. Applying the substitution $p= a^{-\fda} \tau $ we get that
\[
A\equiv\ddt \left[ \int_{\sm (a)}^{t} (t - \tau)^{\al - 1} u_{x}(a, \tau ) d\tau   \right] = - \fda \ddt  \left[ \int_{c_{0}a^{\fda}}^{t} (t - \tau )^{\al-1} f'(\tau a^{- \fda}) \tau a^{-\fda -1} d\tau  \right]
\]
\[
=- \fda a^{\fda - 1} \ddt \left[ \int_{c_{0}}^{ta^{-\fda  }} (t - a^{\fda }p)^{\al -1 } p f'(p) dp  \right] = - \fda a \ddt \left[ \int_{c_{0}}^{ta^{-\fda  }} (ta^{-\fda} - p)^{\al -1 } p f'(p) dp  \right].
\]
After integrating by parts we obtain
\[
A \equiv - \fda a \ddt \left[ \int_{c_{0}}^{ta^{-\fda  }} \frac{(t  a^{-\fda }- p)^{\al  }}{\al} \left( f'(p)+ p f''(p)   \right) dp + \frac{(t  a^{-\fda }- c_{0})^{\al  }}{\al} c_{0}f'(c_{0})  \right].
\]
By the continuity of second derivatives of $f$ in $(c_{0}, k^{\frac{2}{\al}}c_{0}) $ we obtain
\[
\lim_{p \nearrow ta^{-\fda}}  \frac{(t  a^{-\fda }- p)^{\al  }}{\al} \left( f'(p)+ p f''(p)   \right)=0.
\]
Therefore, we obtain
\[
A= - \fda a^{1-\fda}  \left[ \int_{c_{0}}^{ta^{-\fda  }} (t  a^{-\fda }- p)^{\al -1 } \left( f'(p)+ p f''(p)   \right) dp + (t  a^{-\fda }- c_{0})^{\al -1 }  c_{0}f'(c_{0}) \right].
\]
Since $f' \in AC([c_{0},k^{\frac{2}{\al}}c_{0}])$ we get
\[
 \lim_{a \nearrow s(t)}\abs{ \int_{c_{0}}^{ta^{-\fda  }} (t  a^{-\fda }- p)^{\al -1 } f'(p) dp} \leq \sup_{p \in [c_{0},k^{\frac{2}{\al}}c_{0}]} \abs{f'(p)}\lim_{a \nearrow s(t)} \frac{1}{\al}(ta^{-\frac{2}{\al}}-c_{0})=0.
\]
Applying these results together with (\ref{sself}) in (\ref{frees}) we obtain that
\eqq{
\fad c_{1} t^{ \fad-1}  =\jga \fda c_{1}^{1-\fda}t^{\fad-1}\lim_{a \nearrow s(t)}    \left[ \int_{c_{0}}^{ta^{-\fda  }} (t  a^{-\fda }- p)^{\al -1 }  p f''(p)   dp + (t  a^{-\fda }- c_{0})^{\al -1 }  c_{0}f'(c_{0}) \right].
}{fxia}
We note that
\[
\int_{c_{0}}^{ta^{-\fda  }} (t  a^{-\fda }- p)^{\al -1 }pf''(p )dp = - \int_{c_{0}}^{ta^{-\fda  }} (t  a^{-\fda }- p)^{\al }  f''(p )dp + t  a^{-\fda }\int_{c_{0}}^{ta^{-\fda  }} (t  a^{-\fda }- p)^{\al-1 }  f''(p )dp.
\]
Moreover,
\[
\lim_{a \nearrow s(t)} \abs{\int_{c_{0}}^{ta^{-\fda  }} (t  a^{-\fda }- p)^{\al }  f''(p )dp} \leq
\lim_{a \nearrow s(t)}(t  a^{-\fda }- c_{0})^{\al }\int_{c_{0}}^{ta^{-\fda  }}\abs{f''(p)}dp = 0.
\]
Making use of this convergence in (\ref{fxia}), we obtain
\[
\fadk c_{1}^{\fda} \ga  =c_{0}\lim_{a \nearrow s(t)}    \left[ \int_{c_{0}}^{ta^{-\fda  }} (t  a^{-\fda }- p)^{\al -1 }   f''(p)   dp + (t  a^{-\fda }- c_{0})^{\al -1 } f'(c_{0}) \right],
\]
i.e.
\[
\fadk c_{0}^{-2} \ga  =\lim_{b \searrow c_{0}}  \frac{d}{db}  \left[ \int_{c_{0}}^{b} (b- p)^{\al -1 }   f'(p)   dp \right],
\]
where we applied the equality
\eqq{ \int_{c_{0}}^{b} (b- p)^{\al -1 }   f''(p)   dp  = \frac{d}{db}  \left[ \int_{c_{0}}^{b} (b- p)^{\al -1 }   f'(p)   dp \right] -(b-c_{0})^{\al-1}f'(c_{0}).}{fxie}
Thus, we arrive at (\ref{fxic}).
To prove (\ref{fxid}), we notice that from the equation (\ref{eqforf}) we get
\[
 \fdak (\xi - c_{0})^{\al} \xik f''(\xi ) = \frac{(\xi - c_{0})^{\al}  }{\gja} \icxi \xipal f'(p)dp  - \dakda (\xi - c_{0})^{\al}  \xi f'(\xi ) + \jgja.
\]
The function $f'$  is absolutely continuous on some neighborhood of $c_{0}$ thus, taking the limit at $\xi = c_{0}$ we obtain (\ref{fxid}).

\no It remains to show that (\ref{fxic}) implies $f'(c_{0})=0$. We note that
\[
\frac{d}{d b}\int_{c_{0}}^{b}(b-p)^{\al-1}f'(p)dp = \Gamma(\al)\partial^{1-\al}_{c_{0}}f'(b).
\]
We fix $\varepsilon > 0$. Then, from (\ref{fxic}), there exists $x_{0} > c_{0}$ such that for every $x \in (c_{0},x_{0})$
\[
\left(\frac{\al}{2 c_{0}}\right)^{2} - \varepsilon \leq \partial^{1-\al}_{c_{0}}f'(x) \leq \left(\frac{\al}{2 c_{0}}\right)^{2} + \varepsilon.
\]
We note that, since $f'$ is absolutely continuous we have $I^{1-\al}_{c_{0}}\partial^{1-\al}_{c_{0}}f' = f'$. Applying $I^{1-\al}_{c_{0}}$ to the above inequalities we obtain that for every $x \in (c_{0},x_{0})$
\[
\left[\left(\frac{\al}{2 c_{0}}\right)^{2} - \varepsilon\right]\frac{(x-c_{0})^{1-\al}}{\Gamma(2-\al)} \leq f'(x) \leq \left[\left(\frac{\al}{2 c_{0}}\right)^{2} + \varepsilon\right]\frac{(x-c_{0})^{1-\al}}{\Gamma(2-\al)},
\]
hence for every $x \in (c_{0},x_{0})$
\[
\left(\frac{\al}{2 c_{0}}\right)^{2} - \varepsilon \leq f'(x)(x-c_{0})^{\al-1}\Gamma(2-\al) \leq \left(\frac{\al}{2 c_{0}}\right)^{2} + \varepsilon.
\]
The last pair of inequalities is equivalent with
\[
\lim_{x\rightarrow c_{0}}\frac{f'(x)}{(x-c_{0})^{1-\al}} = \left(\frac{\al}{2 c_{0}}\right)^{2}\frac{1}{\Gamma(2-\al)}
\]
and in particular $f'(c_{0})= 0$. This way we finished the proof of Proposition \ref{fprim}.
\end{proof}
We note that, the converse statement also holds. Reverting the calculations, we obtain the following result.

\begin{corollary}\label{zefu}
Assume that  $k > 1$, $c_{0} > 0$  and  function $f$ is such that $f'\in AC([c_{0}, k^{\frac{2}{\al}}c_{0}])$, $f\in C^{2}(c_{0}, k^{\frac{2}{\al}}c_{0})$ and for $\xi \in (c_{0}, k^{\frac{2}{\al}}c_{0})$ the equality (\ref{eqforf}) holds. Then $u(x,t):=f(tx^{-\fda})$ satisfies
\[
\dasx u(x,t) = u_{xx}(x,t) - \jgja (t - \smx )^{-\al} \hd \m{ for   } \hd  s(t)/k<x<s(t), \hd 0<t,
\]
where $s(t)$ is defined in (\ref{sself}) with $c_{1}$ given by (\ref{cjczself}). Furthermore, for every $t~>~0$ there holds  $u_{x}(\cdot , t) \in W^{1,1}(s(t)/k, s(t))$ and   $u_{t}(x,\cdot ) \in AC([\sm(x), \sm(kx)])$ for every $x>0$. If in addition $f$ satisfies (\ref{fini}), then $u(s(t),t)=0$. Moreover, if $f$ satisfies (\ref{fxic})  then $u$ fulfills (\ref{frees}). As a consequence of (\ref{eqforf}) and (\ref{fxic}), (\ref{fxid}) and (\ref{fxidd}) hold and then $u_{x}(s(t),t)=0$.
\end{corollary}

\subsection{Existence of solution}

Now, we shall find the solution to the problem (\ref{eqforf})-(\ref{fxic}). As it was proven in the previous section, if the solution exists, then it also satisfies (\ref{fxidd}) so, it is convenient to consider the space
\[
X_{R}:= \{f\in C^{1}([c_{0}, R]): \hd f(c_{0})=f'(c_{0})=0 \},
\]
for $R\in (c_{0}, \infty)$. Firstly, we transform the equation (\ref{eqforf}) into the weaker form and we obtain the existence of the solution to the transformed equation in the space $X_{R}$.

Let us apply the integral $\icz$  to both sides of (\ref{eqforf})
\[
\ida f'(\xi ) = \fdak \inczxi \tau^{2} f''(\tau ) d \tau  + \dakda \inczxi \tau f'(\tau )d\tau  - \frac{(\xi - c_{0})^{1- \al }}{\Gamma(2-\al)}.
\]
If  we integrate by parts and  take into account that $f(c_{0})=0$, $f'(c_{0})=0$, then  we obtain
\[
\ija f(\xi ) = \dakmda \inczxi  f(\tau ) d \tau  - \dakmda  \xi f(\xi )+\fdak \xi^{2} f'(\xi)  - \frac{(\xi - c_{0})^{1- \al }}{\Gamma(2-\al)}.
\]
We apply again $\icz$ to both sides and integrate by parts to get
\[
\ida f(\xi ) = \dakmda \iczk  f(\xi )   - \tdakmda \inczxi \tau f(\tau ) d\tau +\fdak \xi^{2} f(\xi)  - \frac{(\xi - c_{0})^{2- \al }}{\Gamma(3-\al)}.
\]
The above equality has the following form
\eqq{f(\xi) = Kf(\xi ) + g(\xi ), }{ffred}
where
\[
Kf(\xi) = \fadk \ximk  \ida f(\xi ) + \left[\fad - 1  \right]\ximk  \iczk  f(\xi )   + \left[  3- \fad \right]\ximk \inczxi \tau f(\tau ) d\tau
\]
and
\[
g(\xi ) = \fadk \ximk \frac{(\xi - c_{0})^{2- \al }}{\Gamma(3-\al)}.
\]

\begin{prop}
Assume that $R\in (c_{0}, \infty)$. Then there exists the unique $f\in X_{R}$ solution to (\ref{ffred}).
Furthermore, the obtained solution belongs to $C^{2}((c_{0},R))$ and it satisfies (\ref{eqforf}) on~ $(c_{0},R)$.
\label{existf}
\end{prop}

\begin{proof}
At first, we note that $g\in X_{R}$ and the operator $K$ is linear and bounded on $X_{R}$. After applying Arzeli-Ascoli theorem we deduce that $K$ is compact operator in $X_{R}$ hence, by Fredholm alternative the equation (\ref{ffred}) has the unique solution provided, the homogeneous equation has only one solution. Indeed, from the estimate
\[
|Kf(\xi)|\leq \left[ \fadk c_{0}^{-2} \frac{(\xi - c_{0})^{1-\al}}{\Gamma(2-\al)}+ (1-\fad) c_{0}^{-2} (\xi - c_{0})+  (3-\fad)c_{0}^{-1}  \right]\inczxi |f(\tau)|d\tau
\]
and Gronwall lemma we deduce that the only solution in $X_{R}$ of $f-Kf=0$ is $f\equiv 0$. Since the right hand side of (\ref{ffred}) belongs to $C^{2}((c_{0},R))$, then so does $f$. Hence, we may invert the calculations leading to identity (\ref{ffred}) and we obtain that $f$ satisfies (\ref{eqforf}) on $(c_{0},R)$.
\end{proof}

\begin{prop}
For every $R > 0$ there exists exactly one $f$ belonging to $C^{1}([c_{0},R])\cap C^{2}(c_{0},R)$ which satisfies the system (\ref{eqforf}) - (\ref{fxidd}).
\label{istf}
\end{prop}
\begin{proof}
It remains to show that the solution obtained in Proposition \ref{existf} satisfies (\ref{fxic}) and~(\ref{fxid}). We note that (\ref{fxid}) is a simple consequence of (\ref{eqforf}) and continuity of $f'$.
Let us show~(\ref{fxic}). We fix $\varepsilon > 0$. Then, by (\ref{fxid}) there exists $\xi_{0} > c_{0}$ such that for every $c_{0} <\xi < \xi_{0} $
\[
\left(\frac{\al}{2}\right)^{2}\frac{c_{0}^{-2}}{\Gamma(1-\al)} - \varepsilon \leq (\xi-c_{0})^{\al}f''(\xi) \leq \left(\frac{\al}{2}\right)^{2}\frac{c_{0}^{-2}}{\Gamma(1-\al)} + \varepsilon.
\]
Hence, for every $c_{0} <\xi < \xi_{0} $
\[
\left(\left(\frac{\al}{2}\right)^{2}\frac{c_{0}^{-2}}{\Gamma(1-\al)} - \varepsilon\right)(\xi-c_{0})^{-\al} \leq f''(\xi) \leq \left(\left(\frac{\al}{2}\right)^{2}\frac{c_{0}^{-2}}{\Gamma(1-\al)} + \varepsilon\right)(\xi-c_{0})^{-\al}.
\]
Applying $\frac{1}{\Gamma(1-\al)}I^{\al}_{c_{0}}$ to both these inequalities we obtain that for every $c_{0} <\xi < \xi_{0} $
\[
\left(\frac{\al}{2}\right)^{2}\frac{c_{0}^{-2}}{\Gamma(1-\al)} - \varepsilon \leq \frac{1}{\Gamma(1-\al)} I^{\al}_{c_{0}}f''(\xi) \leq \left(\frac{\al}{2}\right)^{2}\frac{c_{0}^{-2}}{\Gamma(1-\al)} + \varepsilon.
\]
Hence,
\[
I^{\al}_{c_{0}}f''(\xi) \rightarrow \left(\frac{\al}{2}\right)^{2} c_{0}^{-2} \m{ as } \xi \rightarrow c_{0}.
\]
If we recall that $f'(c_{0}) = 0$, then from (\ref{fxie}) we have
\[
\lim_{\xi\rightarrow c_{0}} \frac{d}{d \xi }\int_{c_{0}}^{\xi}(\xi-p)^{\al-1}f'(p)dp = \lim_{\xi\rightarrow c_{0}}\Gamma(\al)I^{\al}_{c_{0}}f''(\xi) = \Gamma(\al)\left(\frac{\al}{2}\right)^{2}c_{0}^{-2}
\]
and we arrive at (\ref{fxic}).
\end{proof}

From Corollary~\ref{zefu} and Proposition~\ref{istf} we deduce the following result.

\begin{corollary}\label{corodwa}
Let $f$ be the solution to (\ref{eqforf})-(\ref{fxidd}) given by Proposition~\ref{istf}. Then, for every $k \in (1,\infty)$ function $u(x,t) := f(tx^{-\fda})$ satisfies
\[
\dasx u(x,t) = u_{xx}(x,t) - \jgja (t - \smx )^{-\al} \hd \m{ for  } \hd  s(t)/k<x<s(t), \hd t>0,
\]
\[
u(s(t),t) =0,
\]
\[
\dot{s}(t) = - \jga  \lim_{a \nearrow s(t)} \ddt \left[ \int_{\sm (a)}^{t} (t - \tau)^{\al - 1} u_{x}(a, \tau ) d\tau   \right],
\]
\[
u_{x}(s(t),t)=0,
\]
where $s(t)$ is defined by (\ref{sself}) with $c_{1}$ given by (\ref{cjczself}). Furthermore, for every $t~>~0$ there hold  $u_{x}(\cdot , t)\in W^{1,1}(s(t)/k, s(t))$ and $u_{t}(x,\cdot ) \in AC([\sm(x), \sm(kx)])$ for every $x>0$.

\end{corollary}

Now, we shall examine the positivity of $u$ given in the above corollary. By (\ref{fxid}) and the equality
\[
f(\xi) = \inczxi (\xi - \tau )f''(\tau) d \tau
\]
we deduce that there exists $\xi_{0}\in (c_{0}, R)$ such that  $f(\xi)>0$ for $\xi \in (c_{0}, \xi_{0})$. In the next subsection we shall show that $f(\xi)>0$ for each $\xi >c_{0}$ and we determine the limit at infinity.

\subsection{Positivity of solution}

We note that it is not clear whether the function $u$ defined in Corollary \ref{corodwa} is positive, because it is not clear whether the function $f$ given in Proposition \ref{istf} is positive. To prove the positivity of $f$ we have to transform the equation (\ref{eqforf}).

\begin{prop}\label{pos}
The function $f$ given in Proposition~\ref{istf} is positive on $(c_{0}, \infty)$. Furthermore,
\eqq{
f(\xi) = \int_{\xi^{-\frac{\al}{2}}}^{c_{1}}\sum_{n=0}^{\infty}(L^{n}G(y))dy,
}{fwzor}
where the constants $c_{0}$ and $c_{1}$ are related by the formula (\ref{cjczself}) and
\eqq{
(Lh)(x):= \frac{1}{\Gamma(1-\al)} \int_{x}^{c_{1}}\int_{\mu}^{c_{1}}(1-p^{-\frac{2}{\al}}
\mu^{\frac{2}{\al}})^{-\al}h(p)dp d\mu,
}{Ldef}
\eqq{
G(x) = \frac{1}{\Gamma(1-\al)} \int_{x}^{c_{1}} (1-c_{0}\mu^{\frac{2}{\al}})^{-\al}d\mu.
}{Gdef}
The series converges uniformly on $[0,c_{1}]$. Moreover, if $F(\mu):=f(\mu^{-\fda})$, then $F\in C^{1}([0,c_{1}])$ and $F''\in L^{1}(0,c_{1})$.
\end{prop}
\begin{proof}
In order to prove the positivity of $f$ on $(c_{0}, \infty)$ we have to transform the equation~(\ref{eqforf}). We introduce $\mu : = \xi ^{-\frac{\al}{2}}$ and
\eqq{
F(\mu):=f(\mu^{-\frac{2}{\al}}) = f(\xi).
}{Fxh}
We note that if $\xi \in (c_{0},\infty)$, then $\mu \in (0,c_{1})$ and $f(c_{0}) = f'(c_{0}) = 0$ implies $F(c_{1}) =F'(c_{1})=0$. We will rewrite the identity (\ref{eqforf}) in terms of function $F$. We note that
\eqq{
F'(\mu) = -\frac{2}{\al}\mu^{-\frac{2}{\al}-1}f'(\mu^{-\frac{2}{\al}})
}{ka}
and
\[
F''(\mu) = \frac{2}{\al}(\frac{2}{\al}+1)\mu^{-\frac{2}{\al}-2}f'(\mu^{-\frac{2}{\al}}) + (\frac{2}{\al})^{2}\mu^{-\frac{2}{\al}-1}\mu^{-\frac{2}{\al}-1}f''(\mu^{-\frac{2}{\al}}).
\]
Hence,
\[
\mu^{2}F''(\mu) = \left[\fdak  + \frac{2}{\al}\right]\xi f'(\xi) + \left(\frac{2}{\al}\right)^{2}\xi^{2}f''(\xi).
\]
Furthermore,
\[
\int_{\mu}^{c_{1}} (\mu^{-\frac{2}{\al}}-p^{-\frac{2}{\al}})^{-\al}F'(p)dp =
-\frac{2}{\al}\int_{\mu}^{c_{1}} (\mu^{-\frac{2}{\al}}-p^{-\frac{2}{\al}})^{-\al}p^{-\frac{2}{\al}-1}f'(p^{-\frac{2}{\al}})dp.
\]
Applying the substitution $p^{-\frac{2}{\al}} = w$ we get
\[
\int_{\mu}^{c_{1}} (\mu^{-\frac{2}{\al}}-p^{-\frac{2}{\al}})^{-\al}F'(p)dp =
-\int_{c_{0}}^{\mu^{-\frac{2}{\al}}}(\mu^{-\frac{2}{\al}} - w)^{-\al}f'(w)dw = -\int_{c_{0}}^{\xi}(\xi - w)^{-\al}f'(w)dw.
\]
Using this calculations in (\ref{eqforf}) we get that function $F$ satisfies
\[
F''(\mu) = -\frac{1}{\Gamma(1-\al)} \mu^{-2} \int_{\mu}^{c_{1}}(\mu^{-\frac{2}{\al}}-p^{-\frac{2}{\al}})^{-\al}F'(p)dp + \frac{1}{\Gamma(1-\al)} \mu^{-2} (\mu^{-\frac{2}{\al}}-c_{0})^{-\al},
\]
which is equivalent with
\eqq{
F''(\mu) = -\frac{1}{\Gamma(1-\al)} \int_{\mu}^{c_{1}}(1-p^{-\frac{2}{\al}}\mu^{\frac{2}{\al}})^{-\al}F'(p)dp + \frac{1}{\Gamma(1-\al)}  (1-c_{0}\mu^{\frac{2}{\al}})^{-\al}.
}{eqforF}
Integrating this equality from $x$ to $c_{1}$ and recalling that $F'(c_{1}) = 0$ we get
\eqq{
F'(x) = \frac{1}{\Gamma(1-\al)} \int_{x}^{c_{1}}\int_{\mu}^{c_{1}}(1-p^{-\frac{2}{\al}}
\mu^{\frac{2}{\al}})^{-\al}F'(p)dp d\mu - \frac{1}{\Gamma(1-\al)} \int_{x}^{c_{1}} (1-c_{0}\mu^{\frac{2}{\al}})^{-\al}d\mu.
}{kb}
We are going to obtain an explicit formula for $F$ and we will show that $F$ is positive in $[0,c_{1}]$. Since $f'$ is continuous in $[c_{0},\infty)$ from (\ref{ka}) we may deduce that $F' \in C(0,c_{1}]$.

Then, identity (\ref{kb}) may be written as
\eqq{
F'(x) = (L F')(x) - G(x)
}{Fxa}
where the operator $L$ and function $G$ are defined by (\ref{Ldef}) and (\ref{Gdef}), respectively.
We apply $L$ to both sides of (\ref{Fxa}) and we deduce that
\[
F'(x) = (L^{2} F')(x) - (G(x)+LG(x)).
\]
Iterating this procedure we obtain that for every $n \in \mathbb{N}$ and every $x \in (0,c_{1})$ we have
\eqq{
F'(x) = (L^{n}F')(x) - \sum_{k=0}^{n}(L^{k}G)(x).
}{kf}
Let us show that for every fixed $x_{0} \in (0,c_{1})$
\eqq{
\lim_{n\rightarrow \infty} \max_{x \in [x_{0},c_{1}]}|(L^{n}F')(x)| = 0.
}{kc}
At first we note that for any $x_{0} \in [0,c_{1}]$ and $h \in C([x_{0},c_{1}])$ there holds
\eqq{
\norm{L^{n}h}_{C([x_{0},c_{1}])} \leq \norm{h}_{C([x_{0},c_{1}])} \norm{L^{n}1}_{_{C([x_{0},c_{1}])}}.
}{ke}
Let us focus on the estimate of $L^{n}1$. By the Fubini theorem we have
\[
\frac{1}{\Gamma(1-\al)} \int_{x}^{c_{1}}\int_{\mu}^{c_{1}}(1-p^{-\frac{2}{\al}}
\mu^{\frac{2}{\al}})^{-\al}dp d\mu  = \frac{1}{\Gamma(1-\al)} \int_{x}^{c_{1}}\int_{x}^{p}(1-p^{-\frac{2}{\al}}
\mu^{\frac{2}{\al}})^{-\al} d\mu dp.
\]
We note that
\eqq{\int_{x}^{p}(1-p^{-\frac{2}{\al}}
\mu^{\frac{2}{\al}})^{-\al} d\mu  \leq \fad B(\fad, 1- \al) p, }{Fxc}
where we applied the substitution $w:= p^{-\frac{2}{\al}}
\mu^{\frac{2}{\al}}$. Hence, we obtain
\eqq{0<L1(x)\leq \frac{\Gamma(1+\frac{\al}{2})}{\Gamma(1-\frac{\al}{2})} c_{1} (I_{x} 1)(c_{1})
\hd \m{ for } \hd x\in [0,c_{1}). }{Fxb}
We shall show by induction that
\eqq{0<L^{n}1(x)\leq \left[ \frac{\Gamma(1+\frac{\al}{2})}{\Gamma(1-\frac{\al}{2})} c_{1}\right]^{n} (I^{n}_{x} 1)(c_{1})
\hd \m{ for } \hd x\in [0,c_{1} )}{Fxd}
for each $n\in \mathbb{N}$. Indeed, suppose that (\ref{Fxd}) holds for $n=k-1$ and then we have
\[
L^{k}1(x) = LL^{k-1}1(x)= \frac{1}{\Gamma(1-\al)} \int_{x}^{c_{1}}\int_{x}^{p} (1-p^{-\frac{2}{\al}}
\mu^{\frac{2}{\al}})^{-\al}   d\mu (L^{k-1}1)(p)  dp
\]
\[
\leq  \frac{1}{\Gamma(1-\al)} \left[ \frac{\Gamma(1+\frac{\al}{2})}{\Gamma(1-\frac{\al}{2})} c_{1}\right]^{k-1} \int_{x}^{c_{1}}\int_{x}^{p} (1-p^{-\frac{2}{\al}}
\mu^{\frac{2}{\al}})^{-\al}   d\mu (I^{k-1}_{p}1)(c_{1} )  dp
\]
\[
\leq \frac{1}{\Gamma(1-\al)} \left[ \frac{\Gamma(1+\frac{\al}{2})}{\Gamma(1-\frac{\al}{2})} c_{1}\right]^{k-1} \int_{x}^{c_{1}}\fad B(\fad, 1- \al) p   (I^{k-1}_{p}1)(c_{1} )  dp,
\]
where in the last inequality we used (\ref{Fxc}). Thus, we have
\[
L^{k}1(x) \leq \left[ \frac{\Gamma(1+\frac{\al}{2})}{\Gamma(1-\frac{\al}{2})} c_{1}\right]^{k} \int_{x}^{c_{1}} (I^{k-1}_{p}1)(c_{1})  dp = \left[ \frac{\Gamma(1+\frac{\al}{2})}{\Gamma(1-\frac{\al}{2})} c_{1}\right]^{k} I^{k}_{x}1(c_{1} )
\]
and (\ref{Fxd}) is proven. We note that
\eqq{(I^{n}_{x}1)(c_{1} ) = \frac{1}{\Gamma(n)} \int_{x}^{c_{1} } (c_{1} - \tau )^{n-1}d \tau = \frac{(c_{1}-x)^{n}}{n!} }{Fxe}
hence, by (\ref{Fxd}) we get
\eqq{0<L^{n}1 (x) \leq \left[ \frac{\Gamma(1+\frac{\al}{2})}{\Gamma(1-\frac{\al}{2})} c_{1}^{2}\right]^{n} \frac{1}{n!}. }{kg}
Applying the estimate (\ref{kg}) in (\ref{ke}) we obtain that
\[
\max_{x \in [x_{0},c_{1}]}|(L^{n}F')(x)| \leq \max_{x \in [x_{0},c_{1}]}|F'(x)| \max_{x \in [x_{0},c_{1}]}|L^{n}1(x)|
\]
\[
\leq \max_{x \in [x_{0},c_{1}]}|F'(x)| \left(\frac{\Gamma(1+\frac{\al}{2})}{\Gamma(1-\frac{\al}{2})} c_{1}^{2} \right)^{n} \frac{1}{n!}
\]
and due to the presence of factorial function in the denominator  the convergence (\ref{kc}) holds.
We will show that the series $\sum_{k=0}^{\infty}(L^{k}G)(x)$ is uniformly convergent on  $[0,c_{1}]$. Indeed, applying the substitution $w:=c_{0}\mu^{\frac{2}{\al}} $ in the definition of $G$ we obtain that
\[
G(x) = \frac{1}{\Gamma(1-\al)}\frac{\al}{2}c_{1}\int^{1}_{c_{0}x^{\frac{2}{\al}}} (1-w)^{-\al}w^{\frac{\al}{2}-1}dw.
\]
Thus,
\[
\max_{x \in [0,c_{1}]}\abs{G(x)} \leq \frac{\Gamma(1+\frac{\al}{2})}{\Gamma(1-\frac{\al}{2})}c_{1}.
\]
Applying estimates (\ref{ke}) and (\ref{kg}) for $x_{0}=0$ we arrive at
\[
\max_{x \in [0,c_{1}]}\abs{L^{n}G(x)} \leq \frac{\Gamma(1+\frac{\al}{2})}{\Gamma(1-\frac{\al}{2})}c_{1} \left(\frac{\Gamma(1+\frac{\al}{2})}{\Gamma(1-\frac{\al}{2})}c_{1}^{2}\right)^{n}\frac{1}{n!}=:a_{n}.
\]
We note that
\[
\frac{a_{n+1}}{a_{n}} = \frac{\Gamma(1+\frac{\al}{2})}{\Gamma(1-\frac{\al}{2})}c_{1}^{2}\frac{1}{n+1}
\rightarrow 0 \m{ as } n\rightarrow \infty.
\]
Hence, by comparison criterion and d'Alembert criterion for convergence of the series we obtain that $\sum_{k=0}^{\infty}(L^{k}G)(x)$ is uniformly convergent on $[0,c_{1}]$. Finally, we may pass to the limit in (\ref{kf}) to obtain
\eqq{
F'(x) = -\sum_{n=0}^{\infty}(L^{n}G)(x) \m{ for every } x \in [0,c_{1}],
}{Fxg}
where the right hand side converges uniformly. As a consequence, $F\in C^{1}([0,c_{1}])$ and by (\ref{eqforF}) we get $F''\in L^{1}(0,c_{1})$.

We note that $L^{n}G(x)>0$ for $[0,c_{1} )$ thus,
\eqq{F' <0 \m{ on } [0,c_{1}). }{Fprimneg}
Applying the fundamental theorem of calculus, we may write
\eqq{
F(x) = - \int_{x}^{c_{1}} F'(y)dy = \int_{x}^{c_{1}} \sum_{n=0}^{\infty}(L^{n}G)(y)dy  \hd \m{ for every } x \in [0,c_{1}].
}{Fsum}
Thus, we have obtained that $F$ is positive on $[0,c_{1})$. We recall that the functions $f$ and $F$ are related by the equality (\ref{Fxh}) therefore,  we proved the claim.
\end{proof}

From Corollary~\ref{corodwa} and Proposition~\ref{pos} we arrive at the following conclusion.

\begin{corollary}\label{coroczi}
Let $c_{1}>0$ and $s(t)=c_{1}t^{\fad}$. Let us define
\[
u(x,t) := \int_{xt^{-\frac{\al}{2}}}^{c_{1}} \sum_{n=0}^{\infty}(L^{n}G(y))dy \hd \m{ for } \hd x\in [0, s(t)], \hd t>0,
\]
where $L$ and $G$ are given by (\ref{Ldef}) and (\ref{Gdef}), respectively. Then, the above series converges uniformly and for every $n \in \mathbb{N}$ there holds $L^{n}G(y)>0$ for every $y\in [0,c_{1})$. Moreover, $u(x,t)$ satisfies
\[
\dasx u(x,t) = u_{xx}(x,t) - \jgja (t - \smx )^{-\al} \hd \m{ for  } \hd  0<x<s(t),
\]
\[
u(s(t),t) =0,
\]
\[
\dot{s}(t) = - \jga  \lim_{a \nearrow s(t)} \ddt \left[ \int_{\sm (a)}^{t} (t - \tau)^{\al - 1} u_{x}(a, \tau ) d\tau   \right],
\]
\[
u_{x}(s(t),t)=0,
\]
for $t>0$. Finally, from equality $u(x,t) = F(xt^{-\fad})$ we deduce that for every $t>0$  $u(\cdot , t)\in W^{2,1}(0,s(t))$ and for every $x>0$ there holds $u_{t}(x,\cdot ) \in C([\sm(x), \infty))$.

\end{corollary}

\begin{corollary}\label{coroa}
Functions $u$ and $s$ defined in Corollary \ref{coroczi} satisfy $u_{x}<0$, $u_{t}>0$ in $\{(x,t)\in \rr\times (0,\infty): \hd 0<x<s(t)  \}$,
\eqq{
\forall \hd  x > 0 \hd  u_{x}(x,\cdot) \in L^{\infty}(s^{-1}(x),\infty)\cap AC_{loc}([s^{-1}(x),\infty))
}{doz1}
and
\eqq{
\forall \hd t > 0 \hd  u_{t}(\cdot,t) \in L^{1}(0,s(t)) \hd \m{and} \hd  D^{\al}_{s^{-1}(\cdot)}u(\cdot,t) \in L^{1}(0,s(t)).
}{doz2}
In particular, the pair $(u,s)$ satisfies the assumptions (A1) - (A3).
\end{corollary}
\begin{proof}
At first, we recall that
\[
u_{x}(x,t) = t^{-\frac{\al}{2}}F'(\mu), \hd  \hd u_{t}(x,t) = -\fad x t^{-\fad -1} F'(\mu),
\]
where $\mu = xt^{-\frac{\al}{2}}$.
Hence, by (\ref{Fprimneg}) we infer $u_{x}<0$, $u_{t}>0$.
Since, $F' \in C([0,c_{1}])$  and for fixed $x > 0$ function $\mu=\mu(x,t)$ is continuous and bounded with respect to $t$ belonging to $ [s^{-1}(x),\infty)$, we obtain that $u_{x}(x,\cdot) \in L^{\infty}(s^{-1}(x),\infty)\cap C([s^{-1}(x),\infty))$.

Let us show that $u_{x}(x,\cdot)$ is absolutely continuous. We may calculate
\[
u_{x,t}(x,t) = -\frac{\al}{2}t^{-\frac{\al}{2}-1}F'(xt^{-\frac{\al}{2}}) - \frac{\al}{2}t^{-\al}xF''(xt^{-\frac{\al}{2}}).
\]
Hence, for every $t^{*} > 0$
\[
\int_{s^{-1}(x)}^{t^{*}}\abs{u_{x,t}(x,t)}dt = \int_{c_{0}x^{\frac{2}{\al}}}^{t^{*}}\abs{-\frac{\al}{2}t^{-\frac{\al}{2}-1}F'(xt^{-\frac{\al}{2}}) - \frac{\al}{2}t^{-\al}xF''(xt^{-\frac{\al}{2}})}dt.
\]
Applying the substitution $\mu = xt^{-\frac{\al}{2}}$ we get
\[
\int_{s^{-1}(x)}^{t^{*}}\abs{u_{x,t}(x,t)}dt \leq \int_{xt^{*-\frac{\al}{2}}}^{c_{1}}x^{-1}\abs{F'(\mu)}d\mu + \int_{xt^{*-\frac{\al}{2}}}^{c_{1}}x^{\frac{2}{\al}-1}\mu^{1-\frac{2}{\al}}\abs{F''(\mu)}d\mu < \infty,
\]
because from Proposition \ref{pos} we have $F' \in C([0,c_{1}])$, $F'' \in L^{1}(0,c_{1})$.
To prove (\ref{doz2}), we note that for every $t > 0$
\[
\| u(\cdot , t)\|_{L^{1}(0, s(t))}= \int_{0}^{s(t)} u_{t}(x,t)dx = -\frac{\al}{2}\int_{0}^{c_{1}t^{\frac{\al}{2}}} xt^{-\frac{\al}{2}-1}F'(xt^{-\frac{\al}{2}})dx = -\frac{\al}{2} t^{\frac{\al}{2}-1}\int_{0}^{c_{1}}p F'(p)dp < \infty,
\]
because $F' \in C([0,c_{1}])$.
Using this results we obtain further,
\[
\int_{0}^{s(t)}\abs{D^{\al}_{s^{-1}(x)}u(x,t)}dx =\frac{1}{\Gamma(1-\al)}\int_{0}^{s(t)}\int_{s^{-1}(x)}^{t}(t-\tau)^{-\al}u_{t}(x,\tau)d\tau dx
\]
\[
=\frac{1}{\Gamma(1-\al)} \int_{0}^{t} (t-\tau)^{-\al} \int_{0}^{s(\tau)} u_{t}(x,\tau)dx d\tau =
-\frac{1}{\Gamma(1-\al)} \frac{\al}{2}\int_{0}^{c_{1}}p F'(p)dp\int_{0}^{t} (t-\tau)^{-\al}\tau^{\frac{\al}{2}-1}d\tau
\]
\[
=-\frac{\Gamma(1+\frac{\al}{2})}{\Gamma(1-\frac{\al}{2})}\int_{0}^{c_{1}}pF'(p)dp < \infty.
\]
Corollary \ref{coroczi} together with (\ref{doz1}) and (\ref{doz2}) implies that the pair $(u,s)$ satisfies the assumptions (A1) - (A3).
\end{proof}

\subsection{Boundary condition}

By Corollary~\ref{coroczi}, for each $c_{1} > 0$ we have obtained self-similar solution of time-fractional Stefan problem $(u,s)_{c_{1}}$ such that
\eqq{
u(0,t) = \int_{0}^{c_{1}} \sum_{n=0}^{\infty}(L^{n}G(y))dy.
}{kdir}
Now, we address the Dirichlet boundary condition (\ref{diri}). We investigate whether for given $\gamma>0$ it is possible to find $c_{1}>0$ such that  $(u,s)_{c_{1}}$ satisfy (\ref{lab})-(\ref{frees}).

For this purpose  we write explicitly the dependence of solution on $c_{1}$. Recall, that from (\ref{Ldef}), (\ref{Gdef}) and the Fubini theorem  we have
\eqq{(\Lcj h)(y) = \jgja  \inycj \inyp (1- p^{-\fda } \mu^{\fda})^{-\al} d\mu h(p)dp,   }{Lcj}
\eqq{\Gcj (y) = \jgja \inycj (1- c_{1}^{-\fda} \mu^{\fda})^{-\al} d\mu.  }{Gcj}

The next proposition provides the representation (\ref{selfuxt}) of the self-similar solution.

\begin{prop}\label{propfajf}
If $c_{1}$ is positive and $s(t)=c_{1}t^{\fad}$, then for  $t>0$ and $x\in [0, s(t)]$ we have
\eqq{
\int_{xt^{-\frac{\al}{2}}}^{c_{1}} \sum_{n=0}^{\infty}(L^{n}_{c_{1}}G_{c_{1}}(y))dy = \int_{xt^{-\fad}}^{c_{1}} H(p, xt^{-\fad})G_{c_{1}}(p)dp,
}{lnnaha}
where the function $H$ is defined by (\ref{selfdefH})-(\ref{selfdefMn}). Furthermore, $H-1$ is positive on the set $W:=\{(p,x): \hd 0\leq x <p  \}$ and $H$ is continuous on $\overline{W}$.

\end{prop}

\begin{proof}
We will find another recursive formula for $L^{n}_{c_{1}}G_{c_{1}}$.
For $0\leq y \leq p < \infty$ we denote
\eqq{M_{1}(p,y) := \jgja \inyp (1- p^{-\fda } \mu^{\fda})^{-\al} d\mu .  }{defM}
Then, we may write
\[
(\Lcj h)(y) =  \inycj M_{1}(p,y) h (p)dp.
\]
Further, we obtain
\[
(\Lcj^{2} \Gcj) (y) = \inycj M_{1}(p,y)(\Lcj \Gcj) (p) dp = \inycj \int_{y}^{r} M_{1}(p,y)M_{1}(r,p) dp \Gcj (r) dr.
\]
Thus, if we denote
\eqq{M_{2}(r,y):= \int_{y}^{r} M_{1}(p,y)M_{1}(r,p) dp}{defMd}
then,
\[
(\Lcj^{2} \Gcj )(y) =  \inycj M_{2}(p,y) \Gcj  (p)dp.
\]
By induction we obtain
\eqq{ (\Lcj^{n} \Gcj )(y) =  \inycj M_{n}(p,y) \Gcj  (p)dp \hd \hd \m{ for } \hd n\geq 1
 }{defLn}
where we set
\eqq{ M_{n}(p,y):= \int_{y}^{p} M_{1}(a,y)M_{n-1}(p, a) da \hd \hd \m{ for } \hd n\geq 2. }{defMn}
Now, we shall obtain the estimate for $M_{n}$. By (\ref{Fxc}) we get
\eqq{M_{1}(p,y)\leq \fGG p.}{estiM}
Then,
\[
M_{2}(p,y) \leq \left[ \fGG \right]^{2}p\inyp a da \leq \left[ \fGG p \right]^{2} (I_{y}1)(p).
\]
We prove by induction that
\eqq{M_{n}(p,y) \leq \left[ \fGG p \right]^{n}(I^{n-1}_{y}1)(p), \hd \hd n\geq 2.
}{estiMin}
Indeed, if
\[
M_{k}(p,y) \leq \left[ \fGG p \right]^{k}(I^{k-1}_{y}1)(p),
\]
then by (\ref{estiM}) we obtain
\[
M_{k+1}(p,y) \leq \left[ \fGG p \right]^{k} \inyp \fGG a  (I^{k-1}_{a}1)(p) da
\]
\[
\leq \left[ \fGG p \right]^{k+1} \inyp   (I^{k-1}_{a}1)(p) da = \left[ \fGG p \right]^{k+1}  (I^{k}_{y}1)(p)
\]
hence, we arrive at (\ref{estiMin}). Applying (\ref{Fxe}) in (\ref{estiMin}) we get the following estimate
\eqq{M_{n}(p,y) \leq \left[ \fGG p \right]^{n}\frac{p^{n-1}}{(n-1)!} \hd \m{ for } \hd n\geq 2. }{estisumM}
Let us define
\eqq{N(p,y) := \sum_{n=1}^{\infty} M_{n}(p,y), \hd 0\leq y \leq p <\infty. }{defsumM}
If $R>0$, then by (\ref{estisumM}) the series converges uniformly on the set
\eqq{
W_{R} = \{(p,y): \hd 0\leq y \leq p \leq R \}.
}{defWR}
In particular, $N$ is continuous, non-negative and bounded on $W_{R}$ for each $R$ positive.

If we sum over $n$ both sides of (\ref{defLn}), then we get
\eqq{\sum_{n=1}^{\infty} \Lcj^{n} \Gcj (y) = \inycj N(p,y) \Gcj (p) dp. }{LnM}
Therefore, we have
\[
\int_{xt^{-\frac{\al}{2}}}^{c_{1}} \sum_{n=0}^{\infty}(L^{n}_{c_{1}}G_{c_{1}}(y))dy =
\int_{xt^{-\frac{\al}{2}} }^{c_{1}} \Gcj (y) dy + \int_{xt^{-\frac{\al}{2}} }^{c_{1}} \inycj N(p,y) \Gcj (p) dp dy .
\]
If we denote
\eqq{H(p,x) := 1+ \int_{x}^{p} N(p,y)dy \hd \m{ for } \hd 0\leq x \leq p  ,  }{defH}
then after applying Fubini theorem we obtain (\ref{lnnaha}).
\end{proof}

Now, we shall investigate the dependence of the self-similar solution obtained in Corollary~\ref{coroczi} from the parameter $c_{1}$. For this purpose we apply the representation given by Proposition~\ref{propfajf} and we denote
\eqq{F_{c_{1}}(x) = \int_{x}^{c_{1}} H(p,x) \Gcj (p) dp. }{FGreen}
Recall, that the function $H$ is continuous and bounded. We examine the continuity of  the mapping
\eqq{c_{1} \mapsto F_{c_{1}}(x) = \int_{x}^{c_{1}} H(p,x) \Gcj (p) dp .}{conticj}
The precise formulation is stated below.

\begin{prop}\label{gammcj}
Assume that $c_{1}$ is positive. Then for every  $x\in [0, c_{1})$
\eqq{\lim_{\cjk\rightarrow c_{1}} F_{\cjk}(x) = F_{c_{1}}(x).  }{limcj}
Moreover, we have
\eqq{\lim_{c_{1}\searrow 0} F_{c_{1}}(0)=0}{limcjz}
and
\eqq{\lim_{c_{1}\nearrow \infty} F_{c_{1}}(0)=\infty.}{limcjin}
Furthermore, if $\gamma > 0$, then there exists positive $c_{1}$ such that
\eqq{F_{c_{1}}(0)= \int_{0}^{c_{1} } H(p, 0 ) G_{c_{1}}(p)dp = \gamma.  }{selfdiricj}

\end{prop}
\begin{proof}
Let us fix $x\in [0, c_{1})$ and assume that $\cjk>c_{1}$. Then by formula (\ref{FGreen}) we get
\[
F_{\cjk}(x) - F_{c_{1}}(x) = \int_{c_{1}}^{\cjk} H(p,x) G_{\cjk} (p)dp + \int_{x}^{c_{1}}H(p,x) [G_{\cjk}(p) - G_{c_{1}}(p)]dp .
\]
We note that $H$ is bounded on $\{(p,x): \hd 0\leq x \leq p \leq \cjk \}$ and
\[
|G_{\cjk}(p)| \leq \frac{\Gamma(1+\fad)}{\Gamma(1-\fad)} \cjk
\]
hence, the first integral converges to zero, if $\cjk \searrow c_{1}$. Next, we write
\[
G_{\cjk}(p) - G_{c_{1}}(p) = \frac{\al}{2\Gamma(1-\al)} \left[ (\cjk - c_{1})\int_{\cjk^{-\fda} p^{\fda}}^{1} (1-w)^{-\al} w^{\fad - 1} dw + c_{1}\int_{\cjk^{-\fda} p^{\fda}}^{c_{1}^{-\fda} p^{\fda}} (1-w)^{-\al} w^{\fad - 1} dw \right].
\]
The first integral is uniformly bounded by $B(1-\al, \fad)$ hence, the first term converges to zero, if $\cjk \searrow c_{1}$. The second integral also  converges to zero because
\[
\int_{\cjk^{-\fda} p^{\fda}}^{c_{1}^{-\fda} p^{\fda}} (1-w)^{-\al} w^{\fad - 1} dw \leq \sup_{W\subset [0,1], |W|\leq (\frac{\cjk}{c_{1}})^{\fda}-1 } \int_{W}(1-w)^{-\al} w^{\fad - 1} dw \rightarrow 0,
\]
if  $\cjk \searrow c_{1}$. The case $\cjk < c_{1}$ may be shown similarly. Therefore, we obtained  (\ref{limcj}).

To get (\ref{limcjz}) we note that
\[
F_{c_{1}}(0) = \int_{0}^{c_{1}} H(p,0)G_{c_{1}}(p)dp \leq  \| H \|_{L^{\infty}(W_{c_{1}})} \frac{\Gamma(1+\fad)}{\Gamma(1-\fad)}c_{1} \rightarrow 0,
\]
if $c_{1}\searrow 0 $.

Recall that $N$ is non-negative, thus we have
\[
F_{c_{1}}(0) \geq \int_{0}^{c_{1}} G_{c_{1}} (p)dp
 = \frac{\al}{2\Gamma(1-\al)}c_{1}\int_{0}^{c_{1}}\int_{c_{1}^{-\frac{2}{\al}}p^{\frac{2}{\al}}}^{1} (1-w)^{-\al}w^{\frac{\al}{2}-1}dwdp
\]
\[
\geq \frac{\al}{2\Gamma(1-\al)}c_{1}\int_{0}^{c_{1}}\int_{c_{1}^{-\frac{2}{\al}}p^{\frac{2}{\al}}}^{1} (1-w)^{-\al}dwdp
=\frac{\al}{2\Gamma(2-\al)}c_{1}\int_{0}^{c_{1}}(1-c_{1}^{-\frac{2}{\al}}p^{\frac{2}{\al}})^{1-\al}dp
\]
\[
= \frac{\fadk c_{1}^{2}}{\Gamma(2-\al)} B(2-\al, \fad) = \frac{\al \Gamma(1+\fad)}{2\Gamma(2-\fad)}c_{1}^{2}\rightarrow \infty \m{ as } c_{1}\rightarrow \infty
\]
and we proved (\ref{limcjin}).

Finally, it remains to prove that  for each  $\gamma \in (0,\infty)$ there exists $c_{1} \in (0,\infty)$ such that
\[
F_{c_{1}}(0) = \gamma.
\]
From (\ref{limcj}) we deduce the continuity of $(0,\infty)\ni c_{1} \mapsto F_{c_{1}}(0)$. Applying the Darboux property together with (\ref{limcjz}), (\ref{limcjin}) we deduce that this map is onto $(0, \infty)$.
\end{proof}
To prove Theorem \ref{samo}, it remains to collect the obtained results.
\begin{proof}[Proof of Theorem~\ref{samo}]
The result is a direct consequence of Corollary~\ref{coroczi}, Corollary~\ref{coroa}, Proposition~\ref{propfajf} and Proposition~\ref{gammcj}.
\end{proof}

\begin{proof}[Proof of Corollary~\ref{neum}]
We note that Corollary~\ref{neum} is a simple consequence of  Theorem~\ref{samo}. Indeed, from the formula (\ref{selfuxt}) we obtain that
\[
u_{x}(0,t) = -t^{-\frac{\al}{2}}\left[c_{1}\frac{\Gamma(1+\frac{\al}{2})}{\Gamma(1-\frac{\al}{2})} + \int_{0}^{c_{1}}N(p,0)G_{c_{1}}(p)dp\right]=:-t^{-\frac{\al}{2}}g(c_{1}).
\]
Since $N$ is continuous and bounded on $W_{R}$ for every $R>0$ and $G_{c_{1}}$ is continuous with respect to $c_{1}$, we obtain that $g$ is continuous as well.
Furthermore, $g(0) = 0$ and $\lim_{c_{1}\rightarrow \infty}g(c_{1}) = \infty$. Thus, Corollary~\ref{neum} follows from the Darboux property.
\end{proof}

\section{Convergence to classical solution }\label{classical}

This section is devoted to the proof of Theorem~\ref{uniformly}. Let us  fix $c_{1}>0$.
We recall the representation of solution to the system (\ref{lab}) - (\ref{frees}) given in Corollary~\ref{coroczi}:
\[
\sal(t) = c_{1}\tad , \hd
\]
\[
\ual(x,t) = \int_{xt^{-\frac{\al}{2}}}^{c_{1}} \sum_{n=0}^{\infty}(L^{n}_{c_{1},\al}G_{c_{1},\al}(y))dy \hd \m{ for } \hd x\in [0, \sal(t)], \hd t>0,
\]
where we added a subscript $\al$ to emphasize that the solution depends on $\al$.
We rewrite also the formulas (\ref{Gcj})-(\ref{defMn}) with a new subscript $\al$. Then we have
\[
M_{1,\al}(p,y) := \jgja \inyp (1- p^{-\fda } \mu^{\fda})^{-\al} d\mu,
\]
\[
M_{n, \al}(p,y):= \int_{y}^{p} M_{1, \al}(a,y)M_{n-1, \al}(p, a) da \hd \hd \m{ for } \hd n\geq 2
\]
for $0\leq y \leq p $ and
\[
\Gcja (y) = \jgja \inycj (1- c_{1}^{-\fda} \mu^{\fda})^{-\al} d\mu,
\]
\[
(\Lcja^{n} \Gcja )(y) =  \inycj M_{n, \al}(p,y) \Gcja  (p)dp \hd \hd \m{ for } \hd n\geq 1 \hd \m{and } \hd 0\leq y \leq c_{1}.
\]
Now we may pass to the proof of Theorem~\ref{uniformly}.
\begin{proof}[Proof of Theorem~\ref{uniformly}]
We would like to pass to the limit with $\al$ in the formula for $u_{\al}$. Hence, at first
we shall calculate the limit as $\al \nearrow 1$ in the formulas for $M_{n, \al}$, $\Gcja$ and $\Lcja^{n} \Gcja$. After a substitution $q:=p^{-\fda } \mu^{\fda}$ we get
\[
\mja (p, y )= p \fGG - \frac{\ald p}{\Gja } \int_{0}^{p^{-\fda } y^{\fda}} (1- q )^{-\al} q^{\fad-1}dq.
\]
We note that
\[
\lima \frac{\ald p}{\Gja } \int_{0}^{p^{-\fda } y^{\fda}} (1- q )^{-\al} q^{\fad-1}dq =0 \m{ \hd for \hd } 0\leq y <p,
\]
because for $0\leq y<p$ we have
\[
\frac{\ald p}{\Gja } \int_{0}^{p^{-\fda } y^{\fda}} (1- q )^{-\al} q^{\fad-1}dq \leq \frac{\ald p}{\Gja } (1- p^{-\fda } y^{\fda} )^{-\al} \int_{0}^{p^{-\fda } y^{\fda}}  q^{\fad-1}dq
\]
\[
=\frac{y}{\Gja} (1- p^{-\fda } y^{\fda} )^{-\al}  \rightarrow 0 \hd \m{ as } \hd \al \rightarrow 1.
\]
Then we denote
\eqq{\mjj(p,y):=\lima \mja(p,y)=\left\{ \begin{array}{lll} \jd p  & \m{ \hd for \hd } & 0\leq y <p, \\ 0 & \m{ \hd  for \hd }  & 0 <y=p. \end{array}\right.}{classf}
From the definition of $G_{c_{1},\al}$ we infer that the same calculations as for $M_{1,\al}$ lead to
\eqq{\Gcjj(p,y):=\lima \Gcja(p,y)=\left\{ \begin{array}{lll} \jd c_{1}  & \m{ \hd for \hd } & 0\leq y <c_{1}, \\ 0 & \m{ \hd  for \hd }  & y=c_{1}
\end{array}\right.}{classg}
and
\eqq{0\leq \Gcja (y) \leq c_{1}, \hd \m{ for \hd  } y\in [0,c_{1}], \hd \al \in (0,1].}{classk}
From (\ref{estiM}) and (\ref{estisumM}) we deduce that
\eqq{0\leq \mna (p,y) \leq  \frac{p^{2n-1}}{(n-1)!} \hd \m{ for \hd } \al\in (0,1), \hd 0\leq y \leq p, \hd n\geq 1  .  }{classh}
The above estimates allow us to apply the Lebesgue's Dominated Convergence Theorem (LDCT) and we get
\eqq{\mnj(p,y):=\lima \mna(p,y)=\int_{y}^{p} M_{1, 1}(a,y)M_{n-1, 1}(p, a) da \hd \hd \m{ for } \hd  0\leq y \leq p, \hd n\geq 2.   }{classi}
Furthermore, the estimate (\ref{classh}) gives
\eqq{0\leq \mnj (p,y) \leq  \frac{p^{2n-1}}{(n-1)!} \hd \m{ for \hd }    0\leq y \leq p, \hd n\geq 1 .  }{classj}
Applying again  (\ref{classk}), (\ref{classh}) together with LDCT we obtain
\eqq{(\Lcjj^{n} \Gcjj)(y) := \lima  (\Lcja^{n} \Gcja)(y)= \int_{y}^{c_{1}} M_{n, 1}(p, y)\Gcjj(p) dp \hd \hd \m{ for } \hd 0\leq y \leq c_{1}, \hd n\geq 1.  }{clsaal}
Moreover, making use of (\ref{classk}), (\ref{classh}) and (\ref{classj}) we arrive at the following estimate
\eqq{0\leq (\Lcja^{n} \Gcja)(y) \leq \frac{c_{1}^{2n+1}}{n!} \hd \m{ for \hd } 0\leq y\leq c_{1}, \hd  n\geq 0, \hd \al\in (0,1].}{classm}
Taking advantage of (\ref{clsaal}) and (\ref{classm}) we get that
\eqq{\lima \sumn (\Lcja^{n}\Gcja)(y)=\sumn (\Lcjj^{n}\Gcjj)(y) \hd \m{ for \hd } y\in [0, c_{1}].  }{classn}
Let us denote by $\ualf$ the extension by zero of $u_{\al}$ for $x$ belonging to $[c_{1}\tad, c_{1}\tjd]$. Namely,
\eqq{\ualf(x,t)= \left\{ \begin{array}{cll} \ual (x,t) & \m{ for } & \hd t>0, \hd x\in [0,c_{1}\tad ] \\ 0 & \m{  for }  & \hd  t>1, \hd  x\in [c_{1}\tad, c_{1}\tjd] .  \\ \end{array}    \right.}{defualfc}
We introduce the following definition
\eqq{\uj(x,t):= \lima \ualf (x,t), \hd t>0, \hd x\in [0,c_{1}\tjd].  }{claa}
We shall characterize the above limit. If $x\in [0,c_{1}\tjd)$, then we note that
\eqq{\uj(x,t):= \lima \int_{x\tmad }^{c_{1}} \sumn (\Lcja^{n}\Gcja)(y) dy =\int_{x\tmjd }^{c_{1}} \sumn (\Lcjj^{n}\Gcjj)(y) dy,  }{clab}
where we applied (\ref{classm}), (\ref{classn}) together with LDCT. If $x=c_{1}\tjd$, then for $t\geq 1 $ we have $\ualf(c_{1}\tjd, t)=0$ so, $\uj(c_{1}\tjd,t)=0$. Finally, if $x=c_{1}\tjd$ and  $t\in(0,1)  $, then
\[
\ualf(c_{1}\tjd, t ) =\lima \ual (c_{1}\tjd, t ) = \lima \int_{c_{1}t^{\frac{1-\al}{2}}}^{c_{1}} \sumn (\Lcja^{n}\Gcja)(y) dy =0,
\]
where we again applied (\ref{classm}). Therefore, we deduce that
\eqq{\uj(x,t)=\int_{x\tmjd }^{c_{1}} \sumn (\Lcjj^{n}\Gcjj)(y) dy, \hd \m{ for } t>0, \hd x\in [0,c_{1}\tjd].   }{clac}
Our next aim is to prove a uniform convergence of $\ualf$ to $u_{1}$ on every compact subset of $\{(x,t): 0< t < \infty, \hd x \in [0,c_{1}t^{\frac{1}{2}}]\}$.
To this end we fix $0<\tds<\tgs$ and we denote
\[
\Qts = \{(x,t): \hd t\in [\tds, \tgs], \hd x\in [0,c_{1}\tjd] \}.
\]
Then, from $\ual(c_{1}\tad, t )=0$ we deduce that $\ualf \in C(\Qts)$. We shall show that $\ualf$ converges uniformly to $\uj$ on $\Qts$. Let us fix $\ep>0$. Without loss of generality, we may assume that $\ep<2c_{1}^{2}e^{c_{1}^{2}}(1- \tds^{\jd})$ in case of $\tds<1$ and $\ep<2c_{1}^{2}e^{c_{1}^{2}}(1- {\tgs}^{-\jd})$ in case of $\tgs>1$. Then, from (\ref{classm}), (\ref{classn}) and LDCT we deduce that there exists $\alz\in (0,1)$ such that
\eqq{\int_{0}^{c_{1}}   \left| \sumn (\Lcjj^{n}\Gcjj)(y) - \sumn (\Lcja^{n}\Gcja)(y)  \right|dy \leq \frac{\ep}{2} \hd \m{ for all \hd } \al\in (\alz, 1) . }{clad}
To estimate $(*):=| \ualf(x,t)- \uj(x,t)|  $ for $(x,t)\in \Qts$, we have to consider three cases.
\begin{itemize}
\item
Case $x\in [0, c_{1}\tjd]$ and $t\leq 1$. In this case we have $\tds\leq 1$ and we may write
\[
(*)= | \ual(x,t)- \uj(x,t)| = \left| \int_{x\tmad }^{c_{1}} \sumn (\Lcja^{n}\Gcja)(y) dy- \int_{x\tmjd }^{c_{1}} \sumn (\Lcjj^{n}\Gcjj)(y) dy  \right|
\]
\[
\leq  \int_{x\tmad }^{x\tmjd} \sumn (\Lcja^{n}\Gcja)(y) dy+
 \int_{0 }^{c_{1}} \left| \sumn (\Lcja^{n}\Gcja)(y)- \sumn (\Lcjj^{n}\Gcjj)(y) \right| dy
\]
\[
\leq c_{1}e^{c_{1}^{2}} x(\tmjd - \tmad) + \frac{\ep}{2},
\]
where we applied (\ref{classm}) and (\ref{clad}). We define $\al_{1}\in (0,1)$ by the equality $c_{1}^{2}e^{c_{1}^{2}}(1-\tds^{\frac{1-\al_{1}}{2}})= \frac{\ep}{2}$. Then, for $\al\in (\max\{\al_{0}, \al_{1} \}, 1 )$ we have
\[
(*)\leq  c_{1}^{2}e^{c_{1}^{2}} \tjd(\tmjd - \tmad) + \frac{\ep}{2} = c_{1}^{2}e^{c_{1}^{2}} (1 - t^{\frac{1-\al}{2}}) + \frac{\ep}{2}\]
\[
\leq c_{1}^{2}e^{c_{1}^{2}} (1 - \tds^{\frac{1-\al}{2}}) + \frac{\ep}{2} \leq  c_{1}^{2}e^{c_{1}^{2}} (1 - \tds^{\frac{1-\al_{1}}{2}}) + \frac{\ep}{2}=\ep.
\]
\item
Case $x\in [c_{1}\tad, c_{1} \tjd] $ and $t\geq 1$. In this case we have $\tgs\geq 1$ and $(*)= \uj(x,t)$. We define $\al_{2}\in (0,1)$ by the equality $c_{1}^{2}e^{c_{1}^{2}}(1 - {\tgs}^{\frac{\al_{2}-1}{2}}  ) = \frac{\ep}{2} $. Then for $\al\in (\al_{2}, 1)$ we have
\[
(*) = \int_{x\tmjd }^{c_{1}} \sumn (\Lcjj^{n}\Gcjj)(y) dy \leq c_{1}e^{c_{1}^{2}}(c_{1}- x \tmjd ) \leq c_{1}^{2}e^{c_{1}^{2}}(1-t^{\frac{\al-1}{2}})\leq
\]
\[
\leq c_{1}^{2}e^{c_{1}^{2}}(1-{\tgs}^{\frac{\al-1}{2}})\leq c_{1}^{2}e^{c_{1}^{2}}(1-{\tgs}^{\frac{\al_{2}-1}{2}})=\frac{\ep}{2},
\]
where we applied (\ref{classm}).
\item
Case $x\in [0,c_{1}\tad ]$ and $t\geq 1$. In this case we have $\tgs\geq 1$ and
\[
(*)= | \ual(x,t)- \uj(x,t)| = \left| \int_{x\tmad }^{c_{1}} \sumn (\Lcja^{n}\Gcja)(y) dy-  \int_{x\tmjd }^{c_{1}}\sumn (\Lcjj^{n}\Gcjj)(y) dy  \right|
\]
\[
\leq  \int^{x\tmad }_{x\tmjd} \sumn (\Lcja^{n}\Gcja)(y) dy+
 \int_{0 }^{c_{1}} \left| \sumn (\Lcja^{n}\Gcja)(y)-  \sumn (\Lcjj^{n}\Gcjj)(y) \right| dy
\]
\[
\leq c_{1}e^{c_{1}^{2}} x( \tmad - \tmjd ) + \frac{\ep}{2},
\]
where we applied (\ref{classm}) and (\ref{clad}).  Then, for $\al\in (\max\{\al_{0}, \al_{2} \}, 1 )$ we have
\[
(*)\leq  c_{1}^{2}e^{c_{1}^{2}} \tad(  \tmad - \tmjd) + \frac{\ep}{2} = c_{1}^{2}e^{c_{1}^{2}} (1 - t^{\frac{\al-1}{2}}) + \frac{\ep}{2}
\]
\[
\leq c_{1}^{2}e^{c_{1}^{2}} (1 - {\tgs}^{\frac{\al-1}{2}}) + \frac{\ep}{2} \leq c_{1}^{2}e^{c_{1}^{2}} (1 - {\tgs}^{\frac{\al_{2}-1}{2}}) + \frac{\ep}{2} =\ep.
\]
\end{itemize}
We note that in the calculations above the constant $\al_{1} = \al_{1}(t_{*})$ appears only in the case when $t_{*} \leq 1$ and similarly $\al_{2} = \al_{2}(t^{*})$ appears only in the case when $t^{*} \geq 1$. Hence in general case if $t_{*} > 1$ we set $\al_{1} = 0$ and if $t^{*} < 1$ we set $\al_{2} = 0$.
Then we may write that for any $0<t_{*}<t^{*}$ and any $\varepsilon$ small enough, if $\al\in (\max\{\al_{0}, \al_{1}, \al_{2}  \} ,1)$ then
\[
| \ualf (x,t) - \uj(x,t)| \leq \ep \m{ \hd for \hd } (x,t)\in \Qts
\]
and as a consequence, $\uj $ is continuous on $\Qts$.
\no Having proven a uniform convergence of $\ualf$ to $\uj$ we may calculate the limit equation satisfied by $\uj$. We note that from a construction $\ual$ satisfies
\eqq{\Dasmxa \ual (x,t) -\ualxx (x,t) = -\jgja (t -\smxa)^{-\al}  \hd \m{ for \hd } t>0, \hd x\in (0,c_{1}\tad).  }{disa}
We fix $\vp \in C^{\infty}_{0}(\Qts)$ 
and we multiply (\ref{disa}) by $\vp $. Then we integrate the equation over $Q_{s_{\al},t^{*}}$ and we arrive at
\[
\int_{0}^{s_{\al}(t^{*})}  \insats \Dasmxa \ual (x,t) \vp(x,t) dt dx - \int_{0}^{s_{\al}(t^{*})}  \insats \ualxx (x,t) \vp(x,t) dt dx =
\]
\eqq{ - \jgja \int_{0}^{s_{\al}(t^{*})}   \insats  (t- \smxa )^{-\al} \vp(x,t)dt dx. }{disb}
We shall calculate the limit in all the above terms. Firstly, we note that $\ual(x,\smxa)=0$ and by Theorem~\ref{samo} we have $\ual(x,\cdot)\in AC[\smxa, \tgs]$ so we may write
\[
\int_{0}^{s_{\al}(t^{*})}   \insats \Dasmxa \ual (x,t) \vp(x,t) dt dx
\]
\[
 = \jgja \int_{0}^{s_{\al}(t^{*})}   \insats \insat \ta    \ualt (x,\tau) d \tau  \vp(x,t) dt dx
\]
\[
 = \jgja \int_{0}^{s_{\al}(t^{*})}   \insats  \ddt \left[ \insat \ta    \ual (x,\tau) d \tau \right] \vp(x,t) dt dx.
\]
If we integrate by parts and next we apply the Fubini theorem we get that
\[
 -\jgja \int_{0}^{s_{\al}(t^{*})}   \insats    \insat \ta    \ual (x,\tau) d \tau \vp_{t}(x,t) dt dx
\]
\eqq{
=  -\jgja \int_{0}^{s_{\al}(t^{*})}   \insats    \intats \ta     \vp_{t}(x,t) dt\ual (x,\tau) d \tau  dx.
}{doleb}
We note that
\[
\jgja  \intats \ta     \vp_{t}(x,t) dt
\]
\[
= \jgja  \intats \ta   [  \vp_{t}(x,t) - \vp_{t}(x,\tau)] dt+ \vp_{t}(x,\tau)\frac{(\tgs - \tau )^{1-\al}}{\Gamma(2-\al)}
\]
\[
= \jgja  \intats (t-\tau)^{1-\al} \fint_{\tau}^{t} \vp_{tt}(x,s)ds dt+ \vp_{t}(x,\tau)\frac{(\tgs - \tau )^{1-\al}}{\Gamma(2-\al)}
\underset{\al\nearrow 1}{\longrightarrow }\vp_{t}(x,\tau),
\]
because $\lima \jgja =0 $. We may write the expression (\ref{doleb}) in the following way
\[
- \int_{0}^{\infty}\int_{0}^{t^{*}}\chi_{[0,s_{\al}(t^{*})]}(x)\chi_{[s_{\al}^{-1}(x),t^{*}]}(\tau)
\frac{1}{\Gamma(1-\al)}\int_{\tau}^{t^{*}}(t-\tau)^{-\al}{\bar{\varphi}}_{t}(x,t)dt {\bar{u}_{\al}}(x,\tau)d\tau dx,
\]
where $\bar{\varphi}$ and $\bar{u}_{\al}$ denote the extensions of $\varphi$ and $u_{\al}$ by zero on $[0,\infty)\times[0,t^{*}]$.
We recall that by (\ref{classm}) we have
\[
u_{\al}(x,t) \leq c_{1}^{2}e^{c_{1}^{2}} \m{ for any } t> 0 \m{ and } x \in [0,c_{1}t^{\frac{\al}{2}}].
\]
Furthermore, for $\al$ close enough to one, we get
\[
\abs{\frac{1}{\Gamma(1-\al)}\int_{\tau}^{t^{*}}(t-\tau)^{-\al}\varphi_{t}(x,t)dt} \leq 2\abs{\varphi_{t}(x,\tau)} \m{ for any } (x,t) \in Q_{t_{*},t^{*}}.
\]
Hence, we arrive at the following estimate
\[
\abs{\chi_{[0,s_{\al}(t^{*})]}(x)\chi_{[s_{\al}^{-1}(x),t^{*}]}(\tau)
\frac{1}{\Gamma(1-\al)}\int_{\tau}^{t^{*}}(t-\tau)^{-\al}{\bar{\varphi}}_{t}(x,t)dt {{\bar{u}}_{\al}}(x,\tau)}
\]
\[
 \leq 2\abs{\varphi_{t}(x,\tau)}c_{1}^{2}e^{c_{1}^{2}}\chi_{[0,\max\{s_{1}(t^{*}),c_{1}\}]}(x).
\]
Recalling that $\widetilde{u}_{\al}(x,t) \rightarrow u_{1}(x,t)$ for $(x,t) \in Q_{t_{*},t^{*}}$ we may apply LDCT to get
\eqq{
\lima \int_{0}^{s_{\al}(t^{*})}   \insats \Dasmxa \ual (x,t) \vp(x,t) dt dx =-\int_{0}^{s_{1}(t^{*})}   \insjts \vp_{t}(x,\tau)  \uj(x,\tau) d\tau dx.
}{disc}
We apply the Fubini theorem and then integration by parts formula, together with the fact that $u_{\al,x}(s_{\al}(t),t) = u_{\al}(s_{\al}(t),t)=0$
to get
\[
\int_{0}^{s_{\al}(t^{*})}  \insats \ualxx (x,t) \vp(x,t) dt dx = \int_{0}^{t^{*}}\int_{0}^{s_{\al}(t)}u_{\al,xx}(x,t)\varphi(x,t)dxdt
\]
\[
=- \int_{0}^{t^{*}}\int_{0}^{s_{\al}(t)}u_{\al,x}(x,t)\varphi_{x}(x,t)dxdt
= \int_{0}^{t^{*}}\int_{0}^{s_{\al}(t)}u_{\al}(x,t)\varphi_{xx}(x,t)dxdt.
\]
Hence, applying again LDCT we obtain that
\eqq{
\lima \int_{0}^{s_{\al}(t^{*})}  \insats \ualxx (x,t) \vp(x,t) dt dx =\int_{0}^{s_{1}(t^{*})}  \insjts \uj (x,t) \vp_{xx}(x,t) dt dx.
}{disd}
Finally, after integrating by parts we obtain
\[
- \jgja \int_{0}^{s_{\al}(t^{*})}   \insats  (t- \smxa )^{-\al} \vp(x,t)dt dx
\]
\[
=  \frac{1}{\Gamma(2-\al)} \int_{0}^{s_{\al}(t^{*})}   \insats  (t- \smxa )^{1-\al} \vp_{t}(x,t)dt dx.
\]
We note that for every $(x,t) \in Q_{s_{\al},t^{*}}$ there holds
\[
(t- \smxa )^{1-\al}\rightarrow 1.
\]
Hence, applying again LDCT we obtain that
\[
- \jgja \int_{0}^{s_{\al}(t^{*})}   \insats  (t- \smxa )^{-\al} \vp(x,t)dt dx
\]
\[
\underset{\al\nearrow 1}{\longrightarrow } \int_{0}^{s_{1}(t^{*})}   \insjts  \vp_{t}(x,t)dt dx = \int_{0}^{s_{1}(t^{*})}  \vp(x, \tgs ) - \vp(x,\smxj) dx=0,
\]
where the last equality holds, because $\vp$ vanishes on a neighborhood of the boundary $\Qts$. Therefore, taking into account the last equality together with (\ref{disc}) and (\ref{disd}), we deduce that $\uj$ is a solution to heat equation
\eqq{u_{1,t}- u_{1,xx}=0 \hd \m{ in } \hd \Qts,}{dise}
at least in the distributional sense.
To complete the proof of Theorem \ref{uniformly} it remains to prove that $u_{1}$ is given by (\ref{classe}).
We recall that $\uj$ is defined by (\ref{classg}), (\ref{classi}), (\ref{clsaal}) and (\ref{clac}).

\no At first, we will show by induction that
\eqq{M_{n,1}(p,y) = \frac{2p}{4^{n}(n-1)!}(p^{2}-y^{2})^{n-1} \hd \m{ for } \hd 0\leq y <p, \hd  \hd n \in \mathbb{N}.}{zbd}
From (\ref{classf}) we see that the formula in (\ref{zbd}) is fulfilled for $n=1$. Let us fix a natural number $k \geq 2$. We assume that for any $l \in \mathbb{N}$ such that $1\leq l\leq k$ we have
\[
M_{l,1}(p,y) = \frac{2p}{4^{l}(l-1)!}(p^{2}-y^{2})^{l-1}.
\]
Then
\[
M_{k+1,1}(p,y) = \frac{1}{2}\int_{y}^{p}aM_{k,1}(p,a)da = \frac{2p}{2\cdot 4^{k} (k-1)!}\int_{y}^{p}a(p^{2}-a^{2})^{k-1}da.
\]
Applying the substitution $a^{2} = w$ we have
\[
M_{k+1,1}(p,y) =\frac{p}{2\cdot 4^{k} (k-1)!}\int_{y^{2}}^{p^{2}}(p^{2}-w)^{k-1}dw = \frac{2p}{4^{k+1}k!}(p^{2} - y^{2})^{k}
\]
and we arrive at the formula (\ref{zbd}) for $n=k+1$. Thus, by the principle of mathematical induction (\ref{zbd}) is proven.
Let us calculate $L^{n}_{c_{1}}G_{c_{1},1}$. Making use of (\ref{clsaal}) and (\ref{zbd}) we get
\[
L^{n}_{c_{1}}G_{c_{1},1}(y) = \int_{y}^{c_{1}}M_{n,1}(p,y)G_{c_{1},1}(p)dp = \frac{c_{1}}{2}\frac{2}{4^{n}(n-1)!} \int_{y}^{c_{1}}p(p^{2}-y^{2})^{n-1}dp.
\]
Applying the substitution $p^{2} = w$ we have
\[
L^{n}_{c_{1}}G_{c_{1},1}(y)=\frac{c_{1}}{2}\frac{1}{4^{n}(n-1)!} \int_{y^{2}}^{c_{1}^{2}}(w-y^{2})^{n-1}dw = \frac{c_{1}}{2\cdot 4^{n}n!}(c_{1}^{2}-y^{2})^{n} = \frac{c_{1}}{2}\frac{1}{n!}\left(\left(\frac{c_{1}}{2}\right)^{2}-\left(\frac{y}{2}\right)^{2}\right)^{n}.
\]
Hence,
\[
u_{1}(x,t) = \int_{\frac{x}{\sqrt{t}}}^{c_{1}}\sum_{n=0}^{\infty}L^{n}_{c_{1}}G_{c_{1},1}(y) dy =
\frac{c_{1}}{2}\int_{\frac{x}{\sqrt{t}}}^{c_{1}}\sum_{n=0}^{\infty}\frac{1}{n!}\left(\left(\frac{c_{1}}{2}\right)^{2}-\left(\frac{y}{2}\right)^{2}\right)^{n}
\]
\[
= \frac{c_{1}}{2}\int_{\frac{x}{\sqrt{t}}}^{c_{1}} e^{(\frac{c_{1}}{2})^{2}-(\frac{y}{2})^{2}}dy = \frac{c_{1}}{2}e^{(\frac{c_{1}}{2})^{2}}\int_{\frac{x}{\sqrt{t}}}^{c_{1}} e^{-(\frac{y}{2})^{2}}dy.
\]
We substitute $y = 2w$ to get
\[
u_{1}(x,t) =c_{1}e^{(\frac{c_{1}}{2})^{2}}\int_{\frac{x}{2\sqrt{t}}}^{\frac{c_{1}}{2}} e^{-w^{2}}dw.
\]
Setting $a = \frac{c_{1}}{2}$ we arrive at
\[
u_{1}(x,t) =2a e^{a^{2}}\int_{\frac{x}{2\sqrt{t}}}^{a} e^{-w^{2}}dw.
\]
Therefore, the function $\uj$ together with $s_{1} = c_{1}t^{\frac{1}{2}}$ is a self-similar solution to the classical Stefan problem (\ref{classa}) - (\ref{classd}). For a construction to a self-similar solution to the classical Stefan problem we refer to Example~1 in Chap. 1.3 \cite{Andreucci}.
In this way we proved the claim.
\end{proof}

\section{Acknowledgments}
The authors are grateful to Prof. Vaughan Voller for his inspiration and fruitful discussions.
The authors would like to thank also prof. Andrea N. Ceretani for an information about the paper \cite{Ros}. Finally, we would like to thank Prof. Piotr Rybka and Prof. Sabrina Roscani for their valuable remarks. 
The authors were partly supported by the National Sciences Center, Poland through 2017/26/M/ST1/00700 Grant.

\end{document}